\documentclass[oneside,british]{amsart}

\usepackage{amsmath}
\usepackage{amsthm}
\usepackage{xfrac}
\usepackage{txfonts}
\usepackage{indentfirst}
\usepackage{mathrsfs}
\usepackage{mathtools}
\usepackage{amssymb}

\newtheorem{theorem}{Theorem}[section]
\newtheorem{lemma}{Lemma}[section]

\newtheorem{definition}{Definition}[section]

\newtheorem{proposition}{Proposition}[section]
\newtheorem{corollary}{Corollary}[section]

\newtheorem{remark}{Remark}[section]
\newtheorem{claim}{Claim}

\newtheorem*{thm*}{Theorem}
\newtheorem*{mainthm1}{Main Theorem 1}
\newtheorem*{mainthm2}{Main Theorem 2}
\newtheorem*{thmA}{Theorem A}
\newtheorem*{thmB}{Theorem B}
\newtheorem*{thmA'}{Theorem A'}
\newtheorem*{thmB''}{Theorem B''}
\newtheorem*{thmC}{Theorem C}

\numberwithin{equation}{section}

\newcommand{\R}{\mathbb{R}}
\newcommand{\Z}{\mathbb{Z}}
\newcommand{\eps}{\varepsilon}

\title{The high-dimensional Weierstrass functions}
\author{Haojie Ren and Weixiao Shen}
\address{Department of Mathematics, Technion, Haifa, Israel}
\address{Shanghai Center for Mathematical Sciences, Jiangwan Campus, Fudan University, No 2005 Songhu Road, Shanghai, China 200438}
\email{20110180012@fudan.edu.cn, wxshen@fudan.edu.cn}
\begin{document}

\maketitle

\begin{abstract}
For a real analytic periodic function $\phi:\mathbb{R}\to\mathbb{R}^d$, an integer $b \ge 2$ and $\lambda\in(1/b,1)$, we prove that the box dimension and the Hausdorff dimension of the graph of the Weierstrass function $W(x)=\sum_{n=0}^{\infty}{{\lambda}^n\phi(b^nx)}$ are both equal to
$$\min\left\{\log_{\lambda^{-1}}b,\,1+\left(\,d-q\,\right)\left(1+\log_b\lambda\right)\right\},$$
where $q = q(\phi, b, \lambda)$ denotes the maximum dimension of all linear spaces $V < \mathbb{R}^d$ such that the projection $\pi_V W$ is Lipschitz.
\end{abstract}

\section{Introduction}

   We study the graphs of the high-dimensional Weierstrass  functions 
\begin{equation}\label{eqn:Wtype}
   W(x)=W^{\phi}_{\lambda,\,b}(x)=\sum\limits_{n=0}^{\infty}{{\lambda}^n\phi(b^nx)},\quad x \in \mathbb{R},
\end{equation}
    where $b > 1$, $1/b< \lambda < 1$ and $\phi(x):\mathbb{R} \to \mathbb{R}^d$ is  a non-constant $\mathbb{Z}$-periodic Lipschitz function. 
 The famous example by Weierstrass, $\phi(x) = \cos(2\pi x)$, is a continuous nowhere differentiable function (see \cite{Hardy}). Bara{\'n}ski \cite{Baranski} later extended it naturally to the complex case with $\phi(x) = e^{2\pi i x}$. Weierstrass and related functions' graphs are key subjects in fractal geometry, extensively studied since its inception (see \cite{Besicovitch}, \cite[Chapter 5]{Bishorp}, \cite[Section 8.2]{falconer1985geometry}, among others).

  Set $\text{graph} W = \{(x, W(x)) : x \in [0, 1)\}$. Let $\pi_V$ denote the orthogonal projection to the linear space $V < \mathbb{R}^d$. We define $q = q(\phi, \lambda, b)$ as:
\begin{equation}\label{def:q}
     q = \max_{V < \mathbb{R}^d}\left \{\text{dim } V \::\: \pi_V \circ W^{\phi}_{\lambda,\, b} \text{ is Lipschitz}\right\}.
\end{equation}

   % The main results of this paper are the following:

We can now state our main results.
\begin{mainthm1}
Let $b\ge2$ be an integer, $1/b< \lambda < 1$ and
let $\phi:\mathbb{R}\to\mathbb{R}^d$ denote a  $\mathbb{Z}$-periodic real-analytic function.
Then 
\begin{equation}\label{eq:dim}
\text{dim}_B\,\text{graph} W=\text{dim}_H\,\text{graph} W=\min\left\{\log_{\lambda^{-1}}b,\,1+\left(d-q\right)\left(1+\log_b\lambda\right)\right\}.
\end{equation}
\end{mainthm1}
\begin{remark}\label{rem:conut-exa}
 The graph of function $W$ is the attractor of an IFS with contraction rates of $\frac{1}{b}$ and  $\lambda$.
 Thus, we may expect that the right side of (\ref{eq:dim})   equals the affine dimension $D_d= \min\{\log_{\lambda^{-1}} b,\, 1 + d (1+ \log_b \lambda)\}$, see \S \ref{subsec:main}.  But
$q>0$ is indeed possible. (For instance, with a $\mathbb{Z}$-periodic analytic function $W_0$, let $\phi(x)=W_0(x)-\lambda W_0(bx)$, yielding $W^{\phi}_{\lambda,\,b}(x)=W_0(x)$ and $q(\phi,\lambda,b)=d$.) 
The smoothness of $\pi_V W$ in (\ref{def:q}) leads to the dimension decrease in (\ref{eq:dim}).
\end{remark}

Given the integer $b$ and function $\phi$, it is natural to study the values of $q(\phi,\lambda,b)$ under different parameters $\lambda$ in $(1/b,1)$.
Define,
\begin{equation}\nonumber%\label{def:p}
p = p(\phi,b) = \min_{\lambda \in (1/b,1)} q(\phi,\lambda,b).
\end{equation}
For this purpose, we present the following result:
\begin{mainthm2}
	Let  $b\ge 2$ be an integer, and let $\phi:\mathbb{R}\to\mathbb{R}^d$ be a non-constant $\mathbb{Z}$-periodic $C^5$ function. Then
	$p(\phi,b)<d$ and there are at most a finite number of sets $\lambda\in(1/b,1)$ such that $p(\phi,b)<q(\phi,\lambda,b).$
\end{mainthm2}

 These theorems generalize [33 Main Theorem] to the high dimension case.

{\bf Some applications.} We studied the {\em complex Weierstrass function} $W^{e^{2\pi  i x}}_{\lambda,\,b}$, introduced  by Bara{\'n}ski \cite{Baranski}, as a   generalization of the classical Weierstrass function $W^{\cos(2\pi  x)}_{\lambda,\,b}$.
\begin{corollary}\label{cor1}
         Let $\phi(x)=e^{2\pi ix}$. For any integer $b\ge2$ and $\lambda\in(1/b,1)$,  we have
 $$\text{dim}_B\,\text{graph} W=\text{dim}_H\,\text{graph} W=\min\left\{\log_{\lambda^{-1}}b,\,3+2\log_b\lambda\right\}.$$
\end{corollary}

	We typically expect the function $W$ to be so complex that 
$$\text{dim}_B\,\text{graph} W=\text{dim}_H\,\text{graph} W=\min\left\{\log_{\lambda^{-1}}b,\,1+d+d\log_b\lambda\right\}.$$
If $p(\phi,b)=0$, then this will hold for all but finitely many parameters $\lambda\in(1/b,1)$ by Main Theorem 1, 2. Hence, we may anticipate the existence of numerous functions $\phi$  with $p(\phi,b)=0$. The following corollary substantiates our intuition. Set, 
\begin{equation}\label{set:X}
        X=\left\{\phi\in C^{\omega}(\mathbb{R};\,\mathbb{R}^d)\: :\:  \phi\text{ is } \mathbb{Z}\text{-periodic}\right\}.
\end{equation}
Given  intger $r\ge1$, let
$$\parallel f-g\parallel_{r}=\parallel f-g\parallel_{0}+\sum_{k=1}^r\parallel f^{(k)}-g^{(k)}\parallel_{0}\quad\text{ for }f,g\in X.$$
\begin{corollary}\label{cor2}
	 For any integers $b\ge2$ and $r\ge1$, the set 
	\begin{equation}\label{set:Y}
	K=\left\{\phi\in X\::\:p(\phi,b)=0\right\}
	\end{equation}
	is $C^1$ open and $C^{\omega}$ dense in $X$.
\end{corollary}

\medskip
{\bf Historical remarks.}
Let $d=1$.
In the pioneering  work \cite{Besicovitch},  Besicovitch and Ursell  proved that for functions of the form $\sum_{n=0}^{\infty}b_n^{-\alpha}\phi(b_nx)$, the graphs have Hausdorff dimension is $2-\alpha$ if $b_{n+1}/b_n\to\infty$ and $\log b_{n+1}/\log b_n\to1$  (see \cite{baranski2011dimension} for recent developments).
A map as described in (\ref{eqn:Wtype}) is readily observed to be Hölder continuous with an exponent of $\log_b1/\lambda$, thereby implying that the box dimension of its graph is at most $D_1$.  To obtain lower bounds, investigating anti-Holder properties is crucial, see \cite{Kaplan1984, Przytycki1989,  Fraydoun1988}. Following this approach,  Hu and Lau \cite{HuLau1993} fully resolved the box dimension of the Weierstrass function for $d=1$.
Przytycki and Urba\'nski \cite{Przytycki1989} showed that the Hausdorff dimension of the graph of such a $W$ is strictly greater than one.
 In \cite{Mauldin1986}, it was proved that the Hausdorff dimension of the graph of $W$  is bounded below by $D_1-O(1/\log b)$.

The recent research on Weierstrass functions closely relates to Bernoulli convolutions  $\sum_n\pm\gamma^n$, another central topic in fractal geometry.
Solomyak \cite{Solomyak1,Solomyak2} established that for almost every $\gamma\in(1/2,1)$, Bernoulli convolution is absolutely continuous (for algebraic parameters $\gamma$, see Varj\'{u} \cite{VarjuAC}). Combining with Ledrappier \cite{ledrappier1992dimension}, this implies that the Hausdorff dimension of the graph of Takagi functions (with $b = 2$ and $\phi(x)=\text{dist}(x,\mathbb{Z})$ in (\ref{eqn:Wtype})) equals $D_1$ for almost every $\lambda\in(1/2,1)$.
Inspired by \cite{Solomyak2}, Tsujii \cite{tsujii2001fat} studied the absolute continuity of certain specific SRB measures in hyperbolic systems.
Mandelbrot \cite{Mandelbrot1977} conjectured that the Hausdorff dimension of the
graph of $W$ is equal to $D_1$ for $\phi(x)=\cos(2\pi x)$ and all $\lambda\in(1/b,1)$.  Building upon the works of \cite{tsujii2001fat, ledrappier1992dimension}, Bara\'{n}ski et al. \cite{baranski2014dimension} verified this conjecture for integral $b$ and $\lambda$ close to $1$. Through careful transversaility estimates, the conjecture was subsequently proven for integral $b$ and all $\lambda\in(1/b,1)$ in Shen \cite{shen2018hausdorff}. Additionally, refer to \cite{ZhangNa} for the case when $\phi(x)=\sin(2\pi x)$. After the exact dimensional properties of Bernoulli convolutions were proven in \cite{Feng2009} (also see \cite{Feng2023}), Hochman \cite{hochman2014self} demonstrated that the dimension of Bernoulli convolutions is one outside a set of packing dimension zero (see Varj\'{u} \cite{VarjuD} for every transcendental $\gamma$).
Following the insights from \cite{hochman2014self}, the Hausdorff dimension of the graph of the Takagi function has been fully resolved in \cite{barany2019hausdorff}. Developing a nonlinear version of the last work, \cite{ren2021dichotomy} determined the Hausdorff dimension of the graph of the Weierstrass function for integer $b$ and analytic function $\phi:\mathbb{R}\to\mathbb{R}$.

When $d>1$, Bara{\'n}ski \cite{Baranski} introduced the {\em complex Weierstrass function}
$\sum_{n=0}^{\infty}{\lambda}^ne^{b^n2\pi i x},$
and studied the box dimension of its graph when $b$ is an integer and $\lambda$ is close to one.
Motivated by this research,  Ren \cite{Ren2023} explored  the box dimension of the graphs of Weierstrass function for $\mathbb{Z}$-periodic Lipschitz functions $\phi:\mathbb{R} \to \mathbb{R}^2$, considering integer  $b>1$ and sufficiently large $b\lambda^2$. See \cite{Anttila2024} for some discussion for anti-Holder properties.

See \cite{Heurteaux2003, Hunt1998, Romanowska, Schied2024} for randomized Weierstrass functions, and for the graphs of interpolation functions, see \cite{Jiang2023}.

\medskip
{\bf  Organization.} In \S \ref{sec:main}, we introduce our main findings: Theorems A and B. Subsequently, we utilize them to complete the proof of Main Theorems 1 and 2. In \S \ref{sec:LY}, we will introduce Ledrappier-Young theory and state Theorem A', a simplified version of Theorem A. Additionally, we will recall some notation and properties related to entropy, which will be extensively utilized throughout our paper. We will prove Theorem B in \S \ref{sec:proveB}. The rest of the sections are devoted to proving Theorem A'.

\medskip
{\bf Acknowledgment.}  The authors are supported by National Key R$\&$D Program of China, Grant No 2021YFA1003204. Additionally, this work has received support from the New Cornerstone Science Foundation.

\section{Main findings}\label{sec:main}
   In this section, we introduce the main findings of this paper and provide explanations for the ideas presented in these theorems.  Subsequently, 
   we  utilize these findings to prove Main Theorems 1 and 2.

\subsection{Main findings}\label{subsec:main}
Throughout this paper, we fix an integer $b\ge2$.
Let  $\varLambda=\{ 0,1,\ldots,b-1 \}$,
$\varLambda^{\#}=\bigcup_{n=1}^{\infty} \varLambda ^n$ and $\Sigma=\varLambda^{\mathbb{Z}_{+}}$.
 Ledrappier \cite{ledrappier1992dimension} applies ergodic theory techniques in smooth dynamical systems to analyze the graphs of Takagi functions.
  This exploration is motivated by the observation:  the graph of function $W^{\phi}_{\lambda,\,b}$  serves as an invariant repeller within the expanding dynamical system
$T:[0,1)\times\mathbb{R}^d\to[0,1)\times\mathbb{R}^d$,
\begin{equation}\label{def:T}
T(x,y)=\left(bx\mod 1,\frac{y-\phi(x)}{\lambda}\right).
\end{equation}
The slopes of the strong stable manifolds of this system can be expressed as follows: for $\textbf{j}=j_1j_2 j_3 \cdot \cdot \cdot  \in \Sigma $,  set
\begin{equation}\label{def:Y}
Y(x,\textbf{j})=Y_{\lambda ,\,b}^{\phi}(x,\textbf{j})=-\sum\limits_{n=1}^{\infty}{\gamma^{n}\phi^{\prime}\left(\frac x{b^n}+\frac{j_1}{b^n}+\frac{j_2}{b^{n-1}} + \cdot \cdot \cdot + \frac{j_n}b\right)},\,\, x \in \mathbb{R}
\end{equation}
where
\begin{equation}\nonumber
\gamma=\frac{1}{b\lambda}\in \left(\frac1b,1\right).
\end{equation}

  when $d=1$,
  the separation properties of the functions $Y(x,\textbf{j})$
  are particularly important in \cite{baranski2014dimension, shen2018hausdorff, ren2021dichotomy}.
  Specifically, we \cite{ren2021dichotomy} introduced a  soft separation condition  and its degenerated counterpart. To analyze the case 
   $d\ge2$, we need to generalize these notations to high-dimensional spaces.

\begin{definition}\label{def:HH}
	Given $\lambda\in (1/b,1)$,	
	we say that a $\mathbb{Z}$-periodic $C^1$ function $\phi:\mathbb{R}\to \mathbb{R}^d$ satisfies
	\begin{itemize}
		\item the condition (H) if 
        \begin{equation}\label{def:H}
        Y_{\lambda,\,b }^{\pi_A\phi}(x,\,\textbf{j})-Y_{\lambda ,\,b}^{\pi_A\phi}(x,\,\textbf{i}) \nequiv 0, \quad \forall \, \textbf{j} \neq \textbf{i} \in \Sigma
       \end{equation}
		for any one-dimensional linear space $A<\mathbb{R}^d$.
		\item
		the condition (H$^*$) if 
		$$Y_{\lambda,\,b }^{\phi}(x,\,\textbf{j})-Y_{\lambda,\,b }^{\phi}(x,\,\textbf{i}) \equiv 0, \quad \forall \, \textbf{j},\,\, \textbf{i} \in \Sigma.$$
	\end{itemize}
\end{definition}
   The condition (H) is regarded as a form of  separation property. With this notation, we extend \cite[Theorem B]{ren2021dichotomy} from 
    $d=1$ to the case  $d\ge1$. Specifically, we obtain the following.

\begin{thmA}
	If an analytic $\mathbb{Z}$-periodic function $\phi:\mathbb{R}\to\mathbb{R}^d$ satisfies the condition (H) for an integer $b\ge 2$ and $\lambda\in (1/b,1)$, then
	$$\text{dim}_B\,\text{graph} W=\text{dim}_H\,\text{graph} W=\min\left\{\log_{\lambda^{-1}}b,\,1+d+d\log_b\lambda\right\}.$$
\end{thmA}
\medskip

{\bf The ideas of Theorem A.}
$\text{graph} W$   serves as the invariant repeller for the expanding map $T$ defined by (\ref{def:T}). The local inverse of $T$ can be represented as $g_i(x,y)=\left(\frac{x+i}b,\lambda y+\phi(\frac{x+i}b)\right)$ where $i\in\Lambda$. Consequently, the graph of $W$ also acts as the attractor for the nonlinear IFS $\{g_i\}_{i\in\Lambda}$. By Falconer's findings \cite{Falconer}, we anticipate that the dimension of $\text{graph} W$ equals the self-affine dimension, given by $\min\left\{\log_{\lambda^{-1}}b,\,1+d+d\log_b\lambda\right\}$.

Let $\mu$ be the lift of the standard Lebesgue measure on $[0, 1)$ to the graph of $W$. Ledrappier \cite{ledrappier1992dimension} demonstrates, through inverse limits, that for most $\textbf{j}\in\Sigma$, the measures $\pi_{\textbf{j}}\mu$ are exact dimensional, with $\text{dim} (\pi_{\textbf{j}}\mu)$ equal to a constant $\alpha$. Here, $\pi_{\textbf{j}}$ is the projection along the strong unstable manifold of the dynamical system $T$. Hence, it suffices to show that $\alpha=\min\{d,\,\log_{\lambda^{-1}}b\}$, as stated in Theorem \ref{thm:ledrappier}.

B\'{a}r\'{a}ny et al. \cite{barany2019hausdorff} developed powerful methods to study self-affine measures on the plane (for recent advances, see \cite{Hochman2022, Rapaport2023}),  which were used to completely solve the dimension of Takagi function, i.e. $\phi(x)=\text{dist}(x,\,\mathbb{Z})$ and $b=2$.
Building on their ideas, we \cite{ren2021dichotomy} developed  methods to handle the nonlinear case and proved Theorem A for $d=1$.
 However, this approach encounters challenges when $d>1$. Notably, Hochman's Entropy Inverse Theorem \cite[Theorem 2.8]{hochman2014self} plays a crucial role in \cite{ren2021dichotomy}, but its higher-dimensional version \cite[Theorem 2.8]{Hochman2015} becomes more intricate. 
In this paper, we fix a good $\textbf{j}\in\Sigma$. We utilize \cite[Theorem 4.12]{Hochman2015} to establish an Entropy Inverse Theorem for measures $\pi_{\textbf{j}}\mu$, as outlined in Theorem \ref{thm:EntropyInverse}. This theorem yields saturated properties for the majority of $\{\pi_{\textbf{i}}\mu \}_{\textbf{i}\in\Sigma} $.
 However, employing \cite[Theorem 2.8]{Hochman2015} directly only provides local saturated properties for  $\pi_{\textbf{j}}\mu$, which is insufficient for deriving a contradiction in our setting.

Specifically, assuming the claim of Theorem A' (see \S \ref{subsec:LY}) fails, we obtain a sequence of $\textbf{j}\in\Sigma$ such that the corresponding measure convolution entropy does not increase, as shown in Theorem C (see \S \ref{sec:proveA'}). Thus, by applying Theorem \ref{thm:EntropyInverse}, we derive a sequence of linear subspaces $V<\mathbb{R}^d$, such that the majority of measures $\{\pi_{\textbf{i}}\mu \}_{\textbf{i}\in\Sigma} $ satisfy saturated properties with respect to them. Subsequently, we leverage the compactness of linear subspaces to find a contradictory direction $V_0$, where important insights from Hochman's work \cite{Hochman2015} are indispensable.

%Note that for any  $\textbf{j}\in\Sigma$ and $n\in\mathbb{Z}_+$, we have:
%\begin{equation}\label{wq:W.1}
%\pi_{\textbf{j}}\mu=\frac1{b^n}\sum_{\textbf{a}\in\Lambda^n}f_{\textbf{a},\,\textbf{j} }\left(\pi_{\textbf{a}^*\textbf{j} }\mu\right),
%\end{equation}
%where $\textbf{a}^*\textbf{j}=a_n\ldots a_1j_1j_2\ldots$ and $f_{\textbf{a},\,\textbf{j} }(z)=\lambda^nz+\pi_{\textbf{j}}g_{\textbf{u}}(0,0),\forall\,z\in\mathbb{Z}^d$. This make $\pi_{\textbf{j}}\mu$ look like  self-similar measures.

\begin{remark}\label{rem:A}
Given a linear space $V<\mathbb{R}^d$, (\ref{eqn:Wtype}) and (\ref{def:Y}) imply that
\[
\pi_V \circ W^{\phi}_{\lambda,\,b} = W^{\pi_V\phi}_{\lambda,\,b} \quad \text{and} \quad \pi_V \circ Y^{\phi}_{\lambda,\,b}(x,\,\textbf{j}) \equiv Y^{\pi_V\phi}_{\lambda,\,b}(x,\,\textbf{j}).
\]
Thus, we can regard $\pi_V \circ W^{\phi}_{\lambda,\,b}$ as a Weierstrass function with kernel function $\pi_V\phi$, and similarly for $\pi_V \circ Y^{\phi}_{\lambda,\,b}(x,\,\textbf{j})$. In particular, we can apply results in \cite{ren2021dichotomy} to $\pi_V \circ W^{\phi}_{\lambda,\,b}$ when $\text{dim}V=1$. We will often use this argument without detailed explanation.
\end{remark}

\begin{proof}[Proof of Corollary \ref{cor1}]
	For $\theta\in\mathbb{R}$, define the projection function $\hat{\pi}_{\theta}:\mathbb{C}\to\mathbb{R}$ as follows:
	$$\hat{\pi}_{\theta}(z)=\frac12\left(ze^{\theta i}+\overline{z}e^{- \theta i}\right).$$ 
        Considering  $\mathbb{C}$  as a real linear space, it suffices to show that
       $\hat{\pi}_{\theta}\phi(x)=\cos\left(2\pi x+\theta\right)$
       satisfies the condition (H) by Theorem A. This is supported by
       \cite[Corollary 1.1, Theorem A]{ren2021dichotomy}.
\end{proof}

    The following are some equivalent definitions of  $q(\phi,\lambda,b)$ and $p(\phi,b)$ (see Corollary
   \ref{lem:eq}), which provide increased convenience for our analysis.    
     Let
\begin{equation}\label{def:q'}
   q'(\phi,\lambda,b)=\max_{V<\mathbb{R}^d}\left\{\text{dim}V\::\:\pi_V\phi\text{ satisfies the condition }  ( \text{H}^*)\right\}
\end{equation}
    and
\begin{equation}\label{def:p'}
    p'(\phi,b)=\min_{\lambda\in(1/b,1)}q'(\phi,\lambda,b).
\end{equation}
    The following is another main finding in our paper.
\begin{thmB}
	Let $b\ge2$ be an integer, and let $\phi:\mathbb{R}\to\mathbb{R}^d$ be a non-constant 
	$\mathbb{Z}$-periodic $C^3$ function. Then
	$p'(\phi,b)<d$, and there exist
	 only finitely many $\lambda\in(1/b,1)$ with
	$$p'(\phi,b)<q'(\phi,\lambda,b).$$
\end{thmB}

{\bf The ideas of Theorem B.}
We will begin by presenting an equivalent characterization of the condition (H$^*$) for $d = 1$ using the Fourier coefficients of $\phi$. This ensures that the degenerate parameter $\lambda\in(1/b,1)$  is a common zero of a sequence of real analytic series, as discussed in Proposition \ref{lem:Fourier}. 
Therefore, the degenerate linear subspace  $V<\mathbb{R}^d$ is the common solution space of a sequence of linear equations dependent on $\lambda$.
 Hence, we utilize theories from linear algebra and complex analysis to complete the proof of Theorem B.

\subsection{Proof of Main Theorems 1, 2}
   The following \cite[Theorem A]{ren2021dichotomy} provides the dichotomy between condition (H) and condition (H$^*$) for $d=1$.
\begin{theorem} \label{thm:dicho}
	Assume that $\phi:\mathbb{R}\to\mathbb{R}$ is $\Z$-periodic and of class $C^k$, where $k\in \{5,6,\ldots, \infty,\omega\}$. Then exactly one of the following holds:
	\begin{enumerate}
		\item [(i)] $W_{\lambda ,\,b}^{\phi}$ is $C^k$ and $\phi$ satisfies the condition (H$^\ast$);
		\item [(ii)] $W_{\lambda,\,b }^{\phi}$ is not Lipschitz and $\phi$ satisfies the condition (H).
		%the graph of $W$ has Hausdorff dimension $D=2+\log_b\lambda$.
	\end{enumerate}
\end{theorem}

The following is an application of this theorem.
\begin{corollary}\label{lem:eq}
	 Let $\phi:\mathbb{R}\to\mathbb{R}^d$ be a $\Z$-periodic $C^k$ function,
	where  $k\in \{5,6,\ldots, \infty,\omega\}$. Then, for any $\lambda\in(1/b,1)$,  we have
	$$q(\phi,\lambda,b)=q'(\phi,\lambda,b).$$
\end{corollary}

\begin{proof}
   Let $V<\mathbb{R}^d$ be a linear space such that $\pi_V\phi$ satisfies condition (H$^*$) and $\text{dim} V=q'$. Given a basis $\{e_j\}_{1\le j\le q'}$ of $V$, each $\pi_{\mathbb{R}e_j}\phi$ satisfies the condition (H$^*$) due to $\pi_{\mathbb{R}e_j}\circ\pi_V=\pi_{\mathbb{R}e_j}$ and $\pi_{\mathbb{R}e_j}\circ Y^{\phi}_{\lambda,\,b}=Y^{\pi_{\mathbb{R}e_j}\phi}_{\lambda,\,b}$. This implies that $\pi_{\mathbb{R}e_j}\circ W^{\phi}_{\lambda,\,b}$ is $C^k$ by Theorem \ref{thm:dicho} and $\pi_{\mathbb{R}e_j}\circ W^{\phi}_{\lambda,\,b}=W^{\,\pi_{\mathbb{R}e_j}\phi}_{\lambda,\,b}$. Therefore, $\pi_V \circ W^{\phi}_{\lambda,\,b}$ is $C^k$, implying $q' \leq q$ by (\ref{def:q}). Utilizing a similar argument and Theorem \ref{thm:dicho}, we have $q \leq q'$.
\end{proof}

\begin{proof}[Proof of Main Theorem 2]
        Corollary \ref{lem:eq} implies 
	$$q'(\phi,\lambda,b)=q(\phi,\lambda,b)\quad\text{and}\quad p'(\phi,b)=p(\phi,b).$$
	Combining this with  Theorem B,  we conclude that Main Theorem 2 holds.
\end{proof}

  Given  a linear space  $V<\mathbb{R}^d$, denote $V^{\bot}$ by
  the orthogonal complement space of $V$.

\begin{proof}[Proof of Main Theorem 1]
      In the proof of Corollary \ref{lem:eq},  there is a  linear space $V < \mathbb{R}^d$ such that $\text{dim}V=q=q'$; additionally, 
      $W_{\lambda,\,b }^{\pi_V\phi}$ is analytic, and $\pi_{V}\phi$ satisfies (H$^*$).  Combining this with  the
      maximality of $q'$ and Theorem \ref{thm:dicho}, we obtain
       the following claim.
\begin{claim}\label{claim:one-dimen-H}
      For any $0\neq a\in V^{\bot}$, (\ref{def:H}) holds for the linear space $A=\mathbb{R}a$.
\end{claim}

     Without loss generality, we may assume that  $V=\{0\}\times\mathbb{R}^{q}$. Then there is a  function $\psi:\mathbb{R}\to\mathbb{R}^{d-q}$
     such that $\pi_{V^{\bot}}\phi(x)=(\psi(x),0)$. Define the  map $F:[0,1)\times\mathbb{R}^{d}\to[0,1)\times\mathbb{R}^{d}$ by
     $ F(x,y)=\left(x,y-W_{\lambda,\,b }^{\pi_{V}\phi}(x)\right), $  
     thus
     %, $F$ is a smooth homeomorphism, and 
  %   there exists an analytic  $\Z$-periodic  function $\psi:\mathbb{R}\to\mathbb{R}^{d-q}$ such that
    \begin{equation}\label{eq:trans-graph}
    F\left(\text{graph} W_{\lambda,\,b }^{\phi}\right)=\left\{\left(x,W_{\lambda,\,b }^{\psi}(x),0\right)\::\: x\in[ 0,1 )\right\}.
   \end{equation}
 
      Claim \ref{claim:one-dimen-H} implies that $\psi$ satisfies condition (H). Thus we have
$$
\text{dim}_B\,\text{graph} W_{\lambda,\,b }^{\psi}=\text{dim}_H\,\text{graph} W_{\lambda,\,b }^{\psi}=\min\left\{\log_{\lambda^{-1}}b,\,1+\left(d-q\right)\left(1+\log_b\lambda\right)\right\}
$$
by Theorem A. This implies that (\ref{eq:dim}) holds, since $F$ is a smooth homeomorphism and (\ref{eq:trans-graph}) holds.
  Finally, the uniqueness of $V$ follows from the maximality of $q'$.
\end{proof}

\section{Ledrappeier's theorem and  entropy}\label{sec:LY}
   This section serves as preparation for proving Theorem A.   
    %Specifically,  we will introduce Ledrappeier's theorem which studies the locall 
    %dimension   of
    %$\mu$, this will be used to give the lower bound of $\text{dim}_H\text{grah} W$.  
    %Finally we will recall some notations and properties of measure entropy.
 %Ledrappier-Young formula  to study the lower bound of $\text{dim}_H\text{}$.
\subsection{ Ledrappier-Young Theory}\label{subsec:LY}
   % Let $\mu$ denote the pushforward of the Lebesgue measure on $[0, 1)$ to the graph of W, defined by $x \to (x, W(x))$.
   % Ledrappier \cite{ledrappier1992dimension} introduced the   Ledrappier-Young Theory to analyze the dimension of $\mu$,
    %this  strategy was later used in
   In a metric space $X$, a probability measure $\omega$ is called {\em exact-dimensional} if there exists a constant $\beta \geq 0$, such that for $\omega$-a.e. $x$,
 \begin{equation}
 \lim_{r\to0}\frac{\log\omega\left(\mathbf{B}(x,r)\right) } {\log r}=\beta.
 \end{equation} 
  We denote $\text{dim }\omega = \beta$, referring to it as the dimension of $\omega$.
 The Ledrappier--Young theory, discussed in \cite{ledrappier1985metric}, examines the exact dimensional properties of ergodic measures under given smooth diffeomorphisms.
  To apply it to the differential homomorphism, one considers the inverse limit (see e.g. \cite{ledrappier1992dimension, XieQian, Shu2010}).
     
  In this subsection, we fix $\lambda \in \left( \frac{1}{b}, 1 \right)$, and consider a $\mathbb{Z}$-periodic $C^2$ function $\phi: \mathbb{R} \to \mathbb{R}^d$. We denote by $\mu$ the pushforward measure of the Lebesgue measure on $[0, 1)$ to the graph of $W$, defined as $x \mapsto (x, W(x))$. It's worth noting that $\text{supp } \mu = \text{graph} W$ is invariant under the expanding map $T$ defined by (\ref{def:T}), and $\mu$ is ergodic with respect to $T$.
      Let's review the Ledrappier–Young results from \cite{ledrappier1992dimension} regarding $\mu$.

    Recall $\Lambda=\{0,\ldots,b-1\}$ and $\Sigma=\Lambda^{\mathbb{Z}{+}}$. The shift map $\sigma:\Sigma\to\Sigma$ is defined as $(i_1,i_2,\ldots)\mapsto(i_2,i_3,\dots)$. Let $\nu$ denote the uniform probability measure on $\Sigma$, and $\nu^{\,\mathbb{Z}+}$ denote the product (Bernoulli) measure on $\Sigma$. For each $i\in\Lambda$, we define the local inverse of $T$ as
\begin{equation}\nonumber
   g_i(x,y)=\left(\frac{x+i}b,\lambda y+\phi\left(\frac{x+i}b\right)\right).
\end{equation}
 Using the inverse limit technique, we define the "inverse" of $T$ as
\begin{equation}\nonumber
G:[0,1)\times\mathbb{R}^d\times\Sigma\to[0,1)\times\mathbb{R}^d\times\Sigma\quad G(x,y,\textbf{i})=\left(g_{i_1}(x,y), \sigma(\textbf{i})\right).
\end{equation}
Then
$\mu=\frac1b\sum_{i\in\Lambda}g_i\mu.$
By iterating this formula, we obtain
\begin{equation}\label{eq:mu-deco}
\mu=\frac1{b^n}\sum_{\textbf{i}\in\Lambda^n}g_{\textbf{i}}\mu,
\end{equation}
where $g_{\textbf{i}}=g_{i_1}\circ g_{i_2}\circ\ldots g_{i_n}$.
With $\textbf{i} \in \Sigma$, detailed calculations show that
\begin{equation}\nonumber
Dg_{i_1}
\left(
\begin{array}{l}
1\\ Y(x,\textbf{i})
\end{array}
\right)=\frac1b
\left(
\begin{array}{l}
1\\ Y\left(\frac{x+i_1}b,\sigma(\textbf{i})\right)
\end{array}
\right).
\end{equation}
Hence, $Dg_{i_n}g_{i_{n-1}}\ldots g_{i_1}$ contracts the vector $(1, Y(x,\textbf{i}))$ at the exponential rate $-\log b$. Set,
\begin{equation}\label{def:Gamma}
\Gamma_{\textbf{i}}(x)=\Gamma^{\phi}_{\textbf{i}}(x)=\int_0^x Y^{\phi}_{\lambda,\,b}(t,\textbf{i})dt.
\end{equation}
   For every $y \in \mathbb{R}^d$, the function $x \mapsto y + \Gamma_{\textbf{i}}(x)$ defines the integral curve of the vector field $(1, Y(t,\textbf{i}))$,
   passing through $(0,y)$. This establishes a foliation in $[0,1)\times\mathbb{R}^d$ where the leaves are 'parallel'.
   For any $\textbf{i}\in\Sigma$, set
\begin{equation}\label{ProjectionFunction}
\pi_{\textbf{i}}(x,y) =\pi^{\phi}_{\textbf{i}}(x,y)= y -\Gamma^{\phi}_{\textbf{i}}(x), \quad\quad
\text{for }(x,y) \in [0,1)\times \mathbb{R}^d.
\end{equation}
  So, $\pi_{\textbf{i}}$ projects $(x, y)$ onto the line $x = 0$ along the foliation $\{y +\Gamma_{\textbf{i}}(x) \}_{y \in \mathbb{R}^d}$. As  \cite{ren2021dichotomy}, we consider $\pi_{\textbf{i}}$ as the flow projection function for $\textbf{i}$.  Then (\ref{eq:mu-deco}) imlies:
\begin{equation}\label{eqn:self`affine'}
\pi_{\textbf{i}}\mu=\frac{1}{b^n}\sum_{\textbf{a}\in \varLambda^n} \pi_{\textbf{i}} g_{\textbf{a}} \mu.
\end{equation}

 Given $\textbf{i}=i_1i_2\cdots i_n\in \varLambda^{\#}$, let  $|\textbf{i}|=n$ be the lengh of 
    $\textbf{i}$, define the reversed sequence 
\begin{equation}\label{eqn:bfi'}
\textbf{i}^*=i_ni_{n-1}\cdots i_1.
\end{equation}
For    $\textbf{j}\in\Sigma$, define $\textbf{i}\textbf{j}:=i_1\ldots i_nj_1j_2\ldots\in\Sigma$.
Utilizing  the methods of \cite[Lemma 3.1]{ren2021dichotomy}, we obtain the transformation equation:
	\begin{equation}
	\label{TransformA}
	\pi_{\textbf{j}}g_{\textbf{i}}(x,y)=\lambda^{|\textbf{i}|}\pi_{{\textbf{i}}^*\textbf{j}}(x,y)+\pi_{\textbf{j}}g_{\textbf{i}}(0,0)\quad\quad\forall\,(x,y)\in[0,1)\times\mathbb{R}^d.
	\end{equation}
Through direct calculation, we establish the following claim:
	\begin{claim}\label{cla:minus}
		For any $\textbf{i}\in\Lambda^{n}$ and $\textbf{j},\textbf{w}\in\Sigma$, it holds that
		$\parallel \pi_{\textbf{i}\textbf{j}}-\pi_{\textbf{i}\textbf{w}}\parallel_{\infty} =O(\gamma^n).$
	\end{claim}

The following  introduces the Ledrappier-Young formula for $\mu$, which can be derived by   the same methods of \cite[Proposition 2]{ledrappier1992dimension}. Consequently, we omit the detailed proof here.

\begin{theorem}\label{thm:ledrappier}
	Let $\phi:\mathbb{R} \to \mathbb{R}^d$ be a $\mathbb{Z}$-periodic  $C^2$ function, then
	\begin{enumerate}
		\item[(1)] $\mu$ is exact dimensional;
		\item[(2)] there is a constant $\alpha=\alpha(\phi)\in [0,d]$ such that  for $\nu^{\,\mathbb{Z}_+}$-a.e. $\textbf{j}\in \Sigma$, $\pi_{\textbf{j}}\mu$ is exact dimensional with $\text{dim }\pi_{\textbf{j}}\mu=\alpha$.
		Furthermore,
		\begin{equation}
		\label{LedrapperYoungF}
		\text{dim }\mu=1+(1+\log_b\lambda)\alpha.
		\end{equation}
	\end{enumerate}
\end{theorem}

The following is a direct corollary of Theorem \ref{thm:ledrappier}.
\begin{corollary}\label{cor:pro}
	Let $V<\mathbb{R}^d$ be a linear space, and denote $\alpha_V= \alpha(\pi_V\phi)$ from Theorem \ref{thm:ledrappier}. Then, for $\nu^{\,\mathbb{Z}_+}$-a.e. $\textbf{j}\in \Sigma$, we have that $\pi_V( \pi_{\textbf{j}}\mu)$ is exact dimensional with $\text{dim }\pi_V( \pi_{\textbf{j}}\mu)=\alpha_V$.
\end{corollary}

\begin{proof}
	 By combining  (\ref{ProjectionFunction}) with (\ref{def:Gamma})  and (\ref{eqn:Wtype}), we obtain:
	$$\pi_V\circ\pi^{\phi}_{\textbf{j}}\left(x,W^{\phi}(x)\right)=\pi^{\pi_V\phi}_{\textbf{j}}\left(x, W^{\pi_V\phi}(x)\right)\quad\quad\text{for } x\in[0,1).$$
	Consequently,
	$\pi_V( \pi^{\phi}_{\textbf{j}}\mu^{\phi})=\pi^{\pi_V\phi}_{\textbf{j}}(\mu^{\pi_V\phi}).$
	Thus, our claim holds by Theorem \ref{thm:ledrappier}.
\end{proof}

The following lemma provides an upper bound for the dimension of $\text{graph} W$.
\begin{lemma}
	$\overline{\text{dim}}_B\,\text{graph} W\le\min\left\{\log_{\lambda^{-1}}b,\,1+d+d\log_b\lambda\right\}.$
\end{lemma}
\begin{proof}
         From  (\ref{eqn:Wtype}), it follows that $W$ is a $\log_b1/\lambda$-H\"older function. By a standard argument, as demonstrated in \cite[Lemma 5.1.6]{Bishorp}, we have  $\overline{\text{dim}}_B\,\text{graph} W\le1+d+d\log_b\lambda$.
Given $n\ge1$, we can cover the graph of $W$ over $[m/b^n,(m+1)/b^n )$ by a ball with length of $O(\lambda^n)$ for each $0\le m<b^n-1$. Consequently, $\overline{\text{dim}}_B\,\text{graph} W\le \log_{\lambda^{-1}}b$.
\end{proof}

Thus, by applying   the mass distribution principle, we simplify Theorem A to:
\begin{thmA'}
 Let $b\geq 2$ be an integer and $\lambda\in (1/b,1)$.  Let $\phi:\mathbb{R}\to\mathbb{R}^d$ be an analytic $\Z$-periodic function satisfying  condition (H). Then,
$\alpha=\min\{d,\,\log_{1/\lambda}b\}$, where $\alpha$ is the constant defined in Theorem~\ref{thm:ledrappier}.
\end{thmA'}

\subsection{Entropy}\label{subsec:entro}
 In this subsection, we outline fundamental entropy properties that are frequently used in our paper, often without additional mention. These properties can be found in \cite[Section 3.1]{Hochman2015}.

Let $(\Omega, \mathcal{B})$ be a measurable space. Given a probability measure $\omega$ on $\Omega$ and a countable partition $\mathcal{Q}\subset \mathcal{B}$ of $\Omega$, we define the {\em entropy} as follows:
$$H(\omega, \mathcal{Q})=\sum_{Q\in\mathcal{Q}} -\omega(Q) \log_b \omega(Q),$$
where the common convention $0\log 0=0$ is adopted.
We denote by $\mathcal{Q}(x)$ the member of $\mathcal{Q}$ containing $x$. If $\omega(\mathcal{Q}(x))>0$, the conditional measure 
$$\omega_{\mathcal{Q}(x)}(A)=\omega_{x,\,\mathcal{Q}}(A)=\frac{\omega(A\cap \mathcal{Q}(x))}{\omega(\mathcal{Q}(x))}$$
is called a {\em $\mathcal{Q}$-component} of $\omega$. 
 In this paper we always assume that $H(\omega, \mathcal{Q})<\infty$. Note that
 we  have the following upper bound:
\begin{equation}\label{eq:upper-boun-entro}
H(\omega, \mathcal{Q})\le\log_b\# \{ Q\in\mathcal{Q}\::\:\omega(Q)>0 \}.
\end{equation}

Given another countable partition $\mathcal{D} \subset \mathcal{B}$, we define the {\em conditional entropy} as:
$$H(\omega, \mathcal{Q}\mid\mathcal{D})=\sum_{D\in \mathcal{D},\, \omega(D)>0} \omega(D) H(\omega_D, \mathcal{Q}).$$
When $\mathcal{Q}$ is a {\em refinement} of $\mathcal{D}$, meaning
$\mathcal{Q}(x)\subset\mathcal{D}(x)$ for each $x\in \Omega$, we have:
$$H(\omega, \mathcal{Q}\mid\mathcal{D})=H(\omega, \mathcal{Q})-H(\omega, \mathcal{D}).$$
Partitions $ \mathcal{D},\,\mathcal{Q}$ are $C$-{\em commensurable} if each element of $\mathcal{D}$ intersects with at most $C$ elements of $\mathcal{Q}$, vise versa, where $C \ge1$ is a constant.
Notably, 
when $ \mathcal{D}$ and $\mathcal{Q}$ are $C$-commensurable:
\begin{equation}\label{eq:comparable}
|H(\omega, \mathcal{Q})-H(\omega, \mathcal{D})|=O(C).
\end{equation}
Set,
$$\mathcal{Q}\vee\mathcal{D}:=\left\{Q\cap D\::\:Q\in\mathcal{Q}\text{ and } D\in\mathcal{D}\right\}.$$

 Let $\omega'$ be another Borel probability measure on $\Omega$. For $t \in [0,1]$, we observe the the {\em Concavity of entropy}:
 \begin{equation}\nonumber%\label{eq;conca-entro}
  tH(\omega,\mathcal{Q})+(1-t)H(\omega',\mathcal{Q})\le H(t\omega+(1-t)\omega',\mathcal{Q}),
\end{equation}
and the {\em Concavity of conditional entropy}:
\begin{equation}\nonumber%\label{eq;conca-condi-entro}
  tH(\omega,\mathcal{D}\mid\mathcal{Q})+(1-t)H(\omega',\mathcal{D}\mid\mathcal{Q})\le H(t\omega+(1-t)\omega',\mathcal{D}\mid\mathcal{Q}).
\end{equation}
We collectively refer to the above two well-known inequalities as {\em Concavity}, which will be extensively utilized in our paper.

In contrast, we also have the  {\em Convexity of entropy}:
\begin{equation}\label{lem:convexity}
   H(t\omega+(1-t)\omega',\mathcal{Q})\le \log_b2+tH(\omega,\mathcal{Q})+(1-t)H(\omega',\mathcal{Q}).
\end{equation}

In particular, we shall often consider the case where $\Omega=\mathbb{R}^d$ and $\mathcal{B}$ denotes the Borel $\sigma$-algebra. Let $\mathcal{L}_n$ denote the partition of $\mathbb{R}^d$ into $b$-adic intervals of level $n$, that is: 
\begin{equation}\label{def:partitionR}
\mathcal{L}_n=\left\{
\left[ \frac{j_1}{b^n},\,\frac{j_1+1}{b^n}\right)\times\ldots\times\left[ \frac{j_d}{b^n},\,\frac{j_d+1}{b^n}\right)\::\:j_1,\,j_2,\ldots j_d\in\mathbb{Z}\right\}.
\end{equation}

 There is a constant $C>0$ with the following properties:
 for any affine map  $f(x)=ax+c$, where $a>0$, $c\in\mathbb{R}^d$ and for any $n\in\mathbb{N}$, we have
	\begin{equation}\label{lem:affinetransform}
	\left|H(f\omega,\mathcal{L}_{n+[\log_b a]})-H(\omega,\mathcal{L}_{n})\right|\le C;
	\end{equation}
given measurable functions  $g,g':\Omega\to\mathbb{R}^d$  with $\sup_x|g(x)-g'(x)|\le b^{-n}$, then
\begin{equation}\label{eq:ero-minus}
\left|H(g\omega, \mathcal{L}_{n})-H(g'\omega, \mathcal{L}_{n})\right|\le C.
\end{equation}

 We denote by $\mathscr{P}(\R^d)$ the set of all compactly supported  Borel probability measures on $\R^d$. For an exact dimensional probability measure $\omega\in \mathscr{P}(\mathbb{R}^d)$, its dimension is closely related to the entropy, as demonstrated in the following fact (cf. \cite[Theorem 4.4]{Young}). See also \cite[Theorem 1.3]{Fan2002}.
\begin{proposition}\label{prop:Young}
	If $\omega \in \mathscr{P}(\mathbb{R}^d)$ is exact dimensional, then
	$$\text{dim }\omega= \lim\limits_{n\to\infty}\frac{1}{n} H(\omega,\mathcal{L}_n).$$
\end{proposition}

%corollary

Utilizing Proposition \ref{prop:Young}, Theorem \ref{thm:ledrappier}, and the Lebesgue Dominated Convergence Theorem yields:
\begin{equation}\label{eq:LCT}
\lim_{n\to\infty}\int\frac{1}{n} H(\pi_{\textbf{j}}\mu,\mathcal{L}_n)d\nu^{\,\mathbb{Z}_+}(\textbf{j})=\alpha.
\end{equation}

\subsection{Component measures}
Let's consider a sequence of partitions $\mathcal{Q}_i$, $i=1,2,\cdots$, where each $\mathcal{Q}_{i+1}$ refines $\mathcal{Q}_i$. We denote $\omega_{x,\,i}:=\omega_{\mathcal{Q}_i(x)}$ as {\em $i$-th component measure} of $\omega$. For a finite set $I$ of positive integers, suppose for each $i\in I$, there exists a random variable $f_i$ defined over $(\Omega, \mathcal{B}(\mathcal{Q}_i), \omega)$, where $\mathcal{B}(\mathcal{Q}_i)$ is the sub-$\sigma$-algebra of $\mathcal{B}$ generated by $\mathcal{Q}_i$. We introduce the following notation:
$$\mathbb{P}_{i\in I} (B_i)=\mathbb{P}_{i\in I}^{\omega}(B_i):=\frac{1}{\# I} \sum_{i\in I} \omega(B_i),$$
where $B_i$ is an event for $f_i$. If $f_i$'s are $\mathbb{R}$-valued random variable, we also use of the notation:
$$\mathbb{E}_{i\in I} (f_i)=\mathbb{E}^{\omega}_{i\in I}(f_i):=\frac{1}{\# I} \sum_{i\in I} \mathbb{E}(f_i).$$
For instance:
\begin{equation}\label{def:Condi-entro}
H(\omega, \mathcal{Q}_{m+n}\mid\mathcal{Q}_n)=\mathbb{E}\left(H(\omega_{x, \,n}, \mathcal{Q}_{m+n})\right)=\mathbb{E}_{i=n} \left(H(\omega_{x,\,i},\mathcal{Q}_{i+m})\right).
\end{equation}
These notations were extensively utilized in ~\cite{hochman2014self,barany2019hausdorff}.

The following useful formulas are derived from \cite[Lemma 3.4]{hochman2014self}. The proof follows the same approach as that of \cite{hochman2014self} and is consequently omitted.
\begin{lemma}\label{lem:minus} 
	Given $R \ge 1$ and $\omega \in \mathscr{P}(\mathbb{R}^d)$ with $\text{diam}\left(\text{supp}(\omega)\right) \le R$, it holds that for all $n \ge m \ge 1$,
	 \begin{equation}\nonumber
	 \begin{split}
	  \frac1{n}H\left(\omega,\mathcal{L}_n\right)&=\frac1n\sum_{i=0}^{n-1}\frac1{m}H\left(\omega,\mathcal{L}_{i+m}\mid\mathcal{L}_{i}\right)+O\left(\frac{m+\log R}n\right)\\&
	  =\mathbb{E}^{\omega}_{0\le i<n}\left(\frac1{m}H(\omega_{y,\,i},\mathcal{L}_{i+m})\right)+O\left(\frac{m+\log R}n\right).
	  \end{split}
	 \end{equation}
\end{lemma}

\section{Entropy inverse theorem}\label{sec:eit}
  In this section, we will utilize Hochman's results \cite[Theorem 4.12]{Hochman2015} to establish an entropy inverse theorem for measures $\pi_{\textbf{j}}\mu$, which will serve as a cornerstone in proving Theorem A'.
 Throughout this section, we fix  $\lambda\in(1/b,1)$ and assume that  $\phi:\mathbb{R} \to \mathbb{R}^d$ is a $\mathbb{Z}$-periodic  $C^2$ function.

Hochman \cite{Hochman2015} introduced the following notations to study the structures of two measures, where the convolution entropy does not increase.

\begin{definition}\label{def:concentrated}
	Let $V<\mathbb{R}^d$ be a linear space and $\eps>0$. A  measure $\omega\in\mathscr{P}(\mathbb{R}^d)$ is said to be $(V,\,\eps)$-{\em concentrated } if there is a translate $V+y$ of $V$ such that $1-\eps$	of the mass of $\omega$ lies within a $\eps$-distance of $V+y$.
\end{definition}

\begin{definition}
	Let $V<\mathbb{R}^d$ be a linear space, $\eps>0$ and $m\ge1$. A  measure $\omega\in\mathscr{P}(\mathbb{R}^d)$ is $(V,\eps,m)$-{\em saturated }, if
	$$\frac1{m}H\left(\omega,\mathcal{L}_m\right)\ge \frac1{m}H\left(\pi_{V^{\bot}}\omega,\mathcal{L}_m\right)+\text{dim}V-\eps.$$
\end{definition}

Recall that $\pi_V$ denotes the orthographic projection function from $\mathbb{R}^d$ onto the subspace $V$, while $V^{\bot}$ represents the orthogonal complement of $V$.
 It's worth noting that $\alpha=\alpha(\phi)$ is defined in Theorem \ref{thm:ledrappier}.
     With these definitions, we can now state the main result of this section.

\begin{theorem}\label{thm:EntropyInverse}
	For any $R\ge 1,\,\eps>0$ and $m\ge M(R,\eps)$, there exists $\delta=\delta(R,\eps,m)>0$ 
	such that 
	 the following holds for $n\ge N(R,\eps,m,\delta)$.
	Consider  a  measure   $\eta\in\mathscr{P}(\mathbb{R}^d)$ with  $\text{diam}\left(\text{supp}(\eta)\right) \le R$.  Suppose there exists $\textbf{j}\in\Sigma$ such that
	\begin{equation}\nonumber
	\frac1{n}H\left(\eta*\pi_{\textbf{j}}\mu,\mathcal{L}_n\right)<\frac1{n}H\left(\pi_{\textbf{j}}\mu,\mathcal{L}_n\right)+\delta
	\end{equation}
	and
	$$ \left|\frac1{n}H\left(\pi_{\textbf{j}}\mu,\mathcal{L}_n\right)-\alpha\right|\le\delta.$$
	Then, there exists a sequence of linear spaces $V_0,V_1,\ldots,V_{n-1}<\mathbb{R}^d$  such that
	\begin{equation}\label{eq:ei2}
	\mathbb{P}^{\eta}_{0\le i<n}\left(\eta_{x,\,i}\text{ is } (V_i,\eps/b^i)-\text{concentrated} \right)>1-\eps
	\end{equation}
	and
	\begin{equation}\label{eq:ei3}
	\frac1n\sum_{i=0}^{n-1}\nu^{\,\mathbb{Z}_+}\left(\left\{\textbf{u}\in\Sigma\::\:\pi_{\textbf{u}}\mu\text{ is } (V_i,\eps,m)-\text{saturated}\right\} \right)>1-\eps.
	\end{equation}	
\end{theorem}
\begin{remark}
 The first assumption is standard in Hochman's entropy inverse theorem \cite[Theorem 2.8]{hochman2014self,Hochman2015}, so we also call  Theorem \ref{thm:EntropyInverse} {\em entropy inverse theorem} for measure $\pi_{\textbf{j}}\mu$.  Additionally, the second condition holds for almost all $\textbf{j}\in\Sigma$ when $n$ is large enough.
\end{remark}

\subsection{Basic properties}The next lemma is from  \cite[Lemma 4.1]{Hochman2015}.
   \begin{lemma}\label{lem:trivialC}
  	For $m\in\mathbb{Z}_+$ and $\eta,\omega\in\mathscr{P}(\mathbb{R}^d)$, we have
  	$$\frac1{m}H\left(\eta*\omega,\mathcal{L}_m\right)\ge\frac1{m}H\left(\omega,\mathcal{L}_m\right)-O(\frac1m).$$
  \end{lemma}

  Given a linear space $V<\mathbb{R}^d$ and $m\in\mathbb{Z}_+$, we define
   \begin{equation}\nonumber%\label{def:LV}
   \mathcal{L}_m^V:=\pi^{-1}_V\mathcal{L}_m
   \end{equation}
    following \cite{Hochman2015}.  For any $\omega\in\mathscr{P}(\mathbb{R}^d)$, we have
    \begin{equation}\label{eq:entro-pull}
    H\left(\pi_V\omega,\mathcal{L}_m\right)=H\left(\omega,\mathcal{L}_m^V\right).
   \end{equation}
It is easy to verify that,
\begin{equation}\label{eq:compa}
    \mathcal{L}_m \text{ and } \mathcal{L}_m^V \vee \mathcal{L}_m^{V^{\bot}} \text{ are } O(1)-\text{commensurable}.
\end{equation}
   This implies:
\begin{equation}\nonumber
\begin{split}   
    H\left(\omega,\mathcal{L}_m\right)&=H\left(\omega,\mathcal{L}_m^V \vee \mathcal{L}_m^{V^{\bot}}\right)+O(1)=H\left(\omega,\mathcal{L}_m^V \vee \mathcal{L}_m^{V^{\bot}}\mid \mathcal{L}_m^{V}\right)+H\left(\omega,\mathcal{L}_m^V\right)+O(1)\\&=\int H\left(\omega_{\mathcal{L}_m^{V}(x)},\mathcal{L}_m^V \vee \mathcal{L}_m^{V^{\bot}}\right)d\omega(x)+H\left(\omega,\mathcal{L}_m^V\right)+O(1)\\&=\int\left( H\left(\omega_{\mathcal{L}_m^{V}(x)},\mathcal{L}_m\right)+O(1)\right)d\omega(x)+H\left(\omega,\mathcal{L}_m^V\right)+O(1).
\end{split}
\end{equation}
Thus we obtain:
\begin{equation}\label{eq:CE}
    H\left(\omega,\mathcal{L}_m\right)=H\left(\omega,\mathcal{L}_m\mid\mathcal{L}_m^V\right)+H\left(\omega,\mathcal{L}_m^V\right)+O(1)\quad\text{for any linear space }V<\mathbb{R}^d.
\end{equation}
This will be frequently used in our paper without further explanation.
If $\omega$ is $(V, \eps, m)$-saturated, meaning $\frac1{m}H\left(\omega,\mathcal{L}_m\right)- \frac1{m}H\left(\pi_{V^{\bot}}\omega,\mathcal{L}_m\right)\ge\text{dim}V-\eps$, then (\ref{eq:CE}) implies:
\begin{equation}\label{eq:satur}
\frac1mH\left(\omega,\mathcal{L}_m\mid\mathcal{L}_m^{V^{\bot}}\right)\ge \text{dim}V-\eps-O(\frac1m).
\end{equation}
      Combining  the concavity of conditional entropy with (\ref{eq:CE}),  the following can be proved using similar methods as in  \cite[Lemma 4.1]{Hochman2015}.
    \begin{lemma}\label{lem:trivialCC}
    	Given a linear space $V<\mathbb{R}^d$,
    	for $m\in\mathbb{Z}_+$ and $\eta,\omega\in\mathscr{P}(\mathbb{R}^d)$, we have
    	$$\frac1{m}H\left(\eta*\omega,\mathcal{L}_m\mid\mathcal{L}_m^V\right)\ge\frac1{m}H\left(\omega,\mathcal{L}_m\mid\mathcal{L}_m^V\right)-O(\frac1m).$$
    \end{lemma}

\subsection{Hochman's results} For $r>0$ and $c,x\in\mathbb{R}^d$, define $S_r(x):=rx$ and $T_c(x):=x+c$. Given a measure
$\omega\in\mathscr{P}(\mathbb{R}^d)$, we denote $r\omega:=S_r\omega$ for convenience. The following is  a direct consequence of \cite[Theorem 4.12]{Hochman2015}.

\begin{theorem}[Hochman]\label{thm:hochman}
For any  $\eps>0$, let $m\ge M(\eps)$. Then there exists a $1\le k\le K(\eps,m)$ such that for sufficiently large $n\ge N(\eps,m,k)$, the following holds. For any $\eta\in\mathscr{P}(\mathbb{R}^d)$ there is a sequence $V_0,\ldots,V_{n-1}$
of subspaces of $\mathbb{R}^d$ such that,	writing $\omega=\eta^{*k}$,
$$\mathbb{P}^{\omega}_{0\le i<n}\left(\,S_{b^i}(\omega_{y,\,i})\text{ is }(V_i,\eps,m)-\text{saturated}\right)>1-\eps$$
and
$$\mathbb{P}^{\eta}_{0\le i<n}\left(\,\eta_{y,\,i}\text{ is }(V_i,\eps/b^i)-\text{concentrated}\right)>1-\eps.$$
\end{theorem}

The following lemma  is \cite[Corollary 4.16]{Hochman2015}.
\begin{lemma}\label{lem:hoc}
	Let $\omega,\,\eta\in\mathscr{P}(\mathbb{R}^d)$. Then
	$$\frac1n H\left(\omega*(\eta^{*k}),\mathcal{L}_n\right)-\frac1n H\left(\omega,\mathcal{L}_n\right)\le k\cdot\left(\frac1n H\left(\omega*\eta,\mathcal{L}_n\right)-\frac1n H\left(\xi,\mathcal{L}_n\right) \right)+O(\frac kn).$$
\end{lemma}

\subsection{Proof of Theorem \ref{thm:EntropyInverse}}
We will frequently utilize the following fundamental fact in probability theory in this paper, details are omitted for brevity.
\begin{lemma}\label{lem:P}
	Let $X$ be a non-negative random variable. For any $M\ge0$ and $\eps>0$, if $\,\mathbb{E}X\le M+\eps$
	and\, $\mathbb{P}\left(X\ge M-\eps\right)\ge 1-\eps$, then
	$$\mathbb{P}\left(X\le M+\sqrt{\eps}\right)\ge 1-O\left((1+M)\sqrt{\eps}\right).$$
\end{lemma}

{\bf Notation.} For $n\in\mathbb{N}$, denote by $\hat{n}$ the unique integer satisfying
\begin{equation}\label{eqn:n'}
\lambda^{\hat{n}}\le b^{-n}< \lambda^{\hat{n}-1}.
\end{equation}

Given   $m\ge1$ and   $\omega\in\mathscr{P}(\mathbb{R}^d)$, for all $i\ge0$ and $y\in \mathbb{R}^d$ such that $\omega({\mathcal{L}_{i}(y)})>0$, we
write,
\begin{equation}\label{ei22}
Q^{\omega}_m(i,y)=\int\frac1m H\left(\omega_{y,\,i}*\lambda^{\hat{i}}(\pi_{\textbf{a}}\mu),\mathcal{L}_{m+i}\right)\,d\nu^{\,\mathbb{Z}_+}(\textbf{a}).
\end{equation}
Recall $\omega_{y,\,i}=\frac1{\omega({\mathcal{L}_{i}(y)})}\omega|_{\mathcal{L}_{i}(y)}$ and $\lambda^{\hat{i}}(\pi_{\textbf{a}}\mu)=S_{\lambda^{\hat{i}}}(\pi_{\textbf{a}}\mu)$, with $S_{\lambda^{\hat{i}}}(x)=\lambda^{\hat{i}}x$.  (\ref{eqn:n'}) implies that
\begin{equation}\label{eq:supp}
\text{The diameter of } \text{supp}\left( \omega_{y,\,i}*\lambda^{\hat{i}}(\pi_{\textbf{a}}\mu)\right) \text{ is bounded by } O(1/{b^i}).
\end{equation}
Utilizing this and (\ref{eq:upper-boun-entro}), we deduce that $Q^{\omega}_m(i,y)=O(1)$.

\begin{lemma}\label{lem:eilow}
     Given $R\ge1$,  consider a measure $\omega\in\mathscr{P}(\mathbb{R}^d)$  with    $\text{diam}\left(\text{supp}(\omega)\right)\le R$.
	For any integers $n>m\ge1$ and  $\textbf{j}\in\Sigma$, it follows that
	$$\frac1{n}H\left(\omega*\pi_{\textbf{j}}\mu,\mathcal{L}_n\right)\ge \mathbb{E}^{\omega}_{0\le i<n}\left(Q^{\omega}_m(i,y)\right)+O\left(\frac1m+\frac{mR}n\right).$$
\end{lemma}

\begin{proof}
%Since $Q^{\,\eta}_m(i,y)=O(1)$, we only need to prove the case where $n/m$ is sufficiently large.
For any $0\le i<n$, let's recall that  $\omega=\mathbb{E}_i^{\omega}(\omega_{y,\,i})$ and  $\pi_{\textbf{j}}\mu=\frac{1}{b^{\hat{i} }  }\sum_{\textbf{z}\in \varLambda^{\hat{i}}} \pi_{\textbf{j}} g_{\textbf{z}} \mu$.  Combining this with Lemma \ref{lem:minus} and  the concavity, we can deduce 
\begin{equation}\nonumber
\begin{split}
\frac1{n}&H\left(\omega*\pi_{\textbf{j}}\mu,\mathcal{L}_n\right)=\frac1n\sum_{i=0}^{n-1}\frac1{m}H\left(\omega*\pi_{\textbf{j}}\mu,\mathcal{L}_{i+m}\mid\mathcal{L}_{i}\right)+ O\left(\frac{mR}n\right)\\&\ge\frac1n\sum_{i=0}^{n-1}\mathbb{E}_{ i }^{\omega}\left(\frac{1}{b^{\hat{i} }}\sum_{\textbf{z}\in \varLambda^{\hat{i}}}\frac1{m}H\left(\omega_{y,\,i}*\pi_{\textbf{j}} g_{\textbf{z}}\mu,\mathcal{L}_{i+m}\mid\mathcal{L}_{i}\right)\right) + O\left(\frac{mR}n\right)
\\&=
\mathbb{E}_{0\le i <n}^{\omega}\left(\frac{1}{b^{\hat{i} }}\sum_{\textbf{z}\in \varLambda^{\hat{i}}}\frac1{m}H\left(\omega_{y,\,i}*\pi_{\textbf{j}} g_{\textbf{z}}\mu,\mathcal{L}_{i+m}\mid\mathcal{L}_{i}\right)\right) + O\left(\frac{mR}n\right).
\end{split}
\end{equation}
 Since  $Q^{\eta}_m(i,y)=O(1)$, we only need to prove the following:

\begin{claim}\label{eq:A.3.2}
For any  $y\in\mathbb{R}^d$and $m,i> 0$ with $\gamma^{\hat{i}}\le b^{-m}$, it follows that
$$\frac{1}{b^{\hat{i} }}\sum_{\textbf{z}\in \varLambda^{\hat{i}}}\frac1{m}H\left(\omega_{y,\,i}*\pi_{\textbf{j}} g_{\textbf{z}}\mu,\mathcal{L}_{i+m}\mid\mathcal{L}_{i}\right)= Q^{\omega}_m(i,y)+ O(\frac1m).$$
\end{claim}

   By    combining
 $\pi_{\textbf{j}} g_{\textbf{z}}(x,y)=\lambda^{\hat{i}}\pi_{\textbf{z}^*\textbf{j}}(x,y)+ \pi_{\textbf{j}} g_{\textbf{z}}(0,0)$  with   (\ref{eq:supp}),  we obtain:
\begin{equation}\label{eq:A.3.1}
\begin{split}
\frac1{m}&H\left(\omega_{y,\,i}*\pi_{\textbf{j}} g_{\textbf{z}}\mu,\mathcal{L}_{i+m}\mid\mathcal{L}_{i}\right)
=\frac1{m}H\left(\omega_{y,\,i}*\lambda^{\hat{i}}(\pi_{\textbf{z}^*\textbf{j}}\mu),\mathcal{L}_{i+m}\right)+O(\frac1m).
\end{split}
\end{equation}
Since $\parallel \pi_{\textbf{z}^*\textbf{j}}-\pi_{\textbf{z}^*\textbf{j}'}\parallel_{\infty} =O(\gamma^{\hat{i}})=O(1/b^m) $ holds for each $\textbf{j}'\in\Sigma$, we can conclude that
$$\frac1{m}H\left(\omega_{y,\,i}*\lambda^{\hat{i}}(\pi_{\textbf{z}^*\textbf{j}}\mu),\mathcal{L}_{i+m}\right)=\frac1{m}H\left(\omega_{y,\,i}*\lambda^{\hat{i}}(\pi_{\textbf{z}^*\textbf{j}'}\mu),\mathcal{L}_{i+m}\right)+O(\frac1m)$$
by (\ref{eq:ero-minus}).  By combining this with (\ref{eq:A.3.1}), we derive the following relationship:
$$\frac1{m}H\left(\omega_{y,\,i}*\pi_{\textbf{j}} g_{\textbf{z}}\mu,\mathcal{L}_{i+m}\mid\mathcal{L}_{i}\right)= \int_{\Sigma}\frac1m H\left(\omega_{y,\,i}*\lambda^{\hat{i}}(\pi_{\textbf{z}^*\textbf{j}'}\mu),\mathcal{L}_{m+i}\right)\,d\nu^{\,\mathbb{Z}_+}(\textbf{j}')+O(\frac1m).$$
Combining this with the definitions of  $\nu^{\,\mathbb{Z}_+}$, we obtain:
\begin{equation}\nonumber
\begin{split}
\frac{1}{b^{\hat{i} }}\sum_{\textbf{z}\in \varLambda^{\hat{i}}}&\frac1{m}H\left(\omega_{y,\,i}*\pi_{\textbf{j}} g_{\textbf{z}}\mu,\mathcal{L}_{i+m}\mid\mathcal{L}_{i}\right) = \frac{1}{b^{\hat{i} }}\sum_{\textbf{z}\in \varLambda^{\hat{i}}}\int\frac1m H\left(\omega_{y,\,i}*\lambda^{\hat{i}}(\pi_{\textbf{z}\textbf{j}'}\mu),\mathcal{L}_{m+i}\right)\,d\nu^{\,\mathbb{Z}_+}(\textbf{j}')+O(\frac1m)\\&=
\int\frac1m H\left(\omega_{y,\,i}*\lambda^{\hat{i}}(\pi_{\textbf{a}}\mu),\mathcal{L}_{m+i}\right)\,d\nu^{\,\mathbb{Z}_+}(\textbf{a})+O(\frac1m).
\end{split}
\end{equation}
Therefore, Claim \ref{eq:A.3.2} holds by the definitions of  $Q^{\omega}_m(i,y)$.
\end{proof}

%\begin{lemma}\label{lem:boundsinm}
%	For any $\varepsilon>0$, $ m\ge M(\varepsilon), n\ge N(\varepsilon,m)$, 
%	$$\inf_{\textbf{j}\in\Sigma}\mathbb{\nu}^{\,n}\left(\left\{\textbf{i}\in \varLambda^{n}\::\: \left|\frac{1}{m}
%	H(\pi_{\textbf{i}\textbf{j}}\mu, \mathcal{L}_{m})-\alpha\right|<\varepsilon \right\}\right) >1-\varepsilon.$$
%\end{lemma}

%\begin{proof}
%Given  linear space $A<\mathbb{R}^d$ with $\text{dim }A=1$, utilizing (\ref{def:H}), Theorem \ref{thm:dicho}, and \cite[Lemma 4.1]{Ren2022}, we find that the measure $\pi_A(\pi^{\phi}_{\textbf{j}}\mu^{\phi})=\pi^{\,\pi_A\phi}_{\textbf{j}}(\mu^{\,\pi_A\phi})$ has no atom for each $\textbf{j}\in\Sigma$. Hence, for any interval $I\subset \mathbb{R}^d$, the function $\textbf{j}\mapsto \pi^{\phi}_{\textbf{j}}\mu^{\phi}(I)$ is continuous. Consequently, we establish the claim using methods similar to \cite[Lemma 4.2]{Ren2022}.
%\end{proof}

When convolution entropy fails to increase, it implies the saturated properties of one measure can be deduced from another. This observation is utilized to derive (\ref{eq:ei3}).  Specifically, we can express it as follows:

\begin{lemma}\label{lem:saturated-trans}
	  For any $\eps > 0$ and $m \geq M(\eps)$, the following holds:  given measures $\omega,\eta \in \mathscr{P}(\mathbb{R}^d)$, 
          if $\omega$ is $(V, \eps, m)$-saturated for some linear subspace $V < \mathbb{R}^d$, and if
	 \begin{equation}\label{eq:6}
	 \frac1m H\left(\omega*\eta,\mathcal{L}_m\right)\le \frac1m H\left(\eta,\mathcal{L}_m\right)+ \eps,
	 \end{equation}
	 then $\eta$ is $(V,3\eps,m)$-saturated.
\end{lemma}

\begin{proof}
        Let $m\ge1$ be such that $O(\frac1m)<\eps/3$. Using
        (\ref{eq:CE})  and  (\ref{eq:6}), we obtain:
\begin{equation}\label{eq:A.3.3}
\frac1mH\left(\eta,\mathcal{L}_m\mid\mathcal{L}^{V^{\bot}}_m\right)+ \frac1m H\left(\eta,\mathcal{L}^{V^{\bot}}_m\right)
\ge   \frac1m H\left(\omega*\eta,\mathcal{L}_m\right)-\frac43\eps.
\end{equation}
   Combining the cancavity with  (\ref{eq:satur}), we get:
	\begin{equation}\nonumber
	\begin{split}
	&\frac1m H\left(\omega*\eta,\mathcal{L}_m\right)
          = \frac1m H\left(\omega*\eta,\mathcal{L}^{V^{\bot}}_m\right)+\frac1m H\left(\omega*\eta,\mathcal{L}_m\mid\mathcal{L}^{V^{\bot}}_m\right)+O(\frac1m)
         \\&\ge
	\frac1m H\left(\eta,\mathcal{L}^{V^{\bot}}_m\right)+\frac1m H\left(\omega,\mathcal{L}_m\mid\mathcal{L}^{V^{\bot}}_m\right)-\frac13\eps
        \ge \frac1m H\left(\eta,\mathcal{L}^{V^{\bot}}_m\right)+\text{dim}V-\frac53\eps.
	\end{split}
    \end{equation}
    By combining this with (\ref{eq:A.3.3}) and(\ref{eq:CE}),
    we conclude that Lemma \ref{lem:saturated-trans} holds.
\end{proof}

{\bf Notations.}
Following  \cite[Section 2.1]{Rapaport}, we introduce the following notations to enhance the convenience of our proofs.
Consider $m\ge1$ and   $k\ge M(m)$,  implying  that $k$ is large with respect to $m$.
We denote this relationship as
$m\ll k$. The relation $\ll$ is transitive.
 Thus, $m\ll k\ll n$  implies that  $k\ge M(m)\ge1$ and $n\ge K(m,\,k)\ge1$.
Similarly, for  $\eps>0$ and $0<\delta<\eps(\eps)$ (i.e., $\delta$ is small with respect to $\eps>0$),  we represent this as $\eps\ll\delta^{-1}$. 

\begin{remark}
The main proof method in the following  consists of repeatedly applying Lemma \ref{lem:P}. Specifically, in each paragraph, we first verify the conditions of Lemma \ref{lem:P}, and then draw conclusions. After repeating this process three times, the proof is completed by Lemma \ref{lem:saturated-trans}.
\end{remark}

\begin{proof}[Proof of Theorem \ref{thm:EntropyInverse}]
Given $\eps > 0$ and $m \ge 1$ such that $1 \le R \ll \eps^{-1} \ll m$, let $K = K(\eps^8, m)\ge1$ as defined in Theorem \ref{thm:hochman}. Set $\delta = \eps^8/K$ and $n \geq 1$ such that $m,K\ll n$.
Let $\textbf{j}\in\Sigma$ and  $\eta\in\mathscr{P}(\mathbb{R}^d)$ satisfy the assumption of  Theorem \ref{thm:EntropyInverse}. Thus, by Theorem \ref{thm:hochman}, there is  
 $1\le k\le K$ and a sequence  $V_0,\ldots,V_{n-1}<\mathbb{R}^d$
such that (\ref{eq:ei2}) holds and
\begin{equation}\label{eq:A.3.10}
\mathbb{P}^{\omega}_{0\le i<n}\left(S_{b^i}(\omega_{y,\,i})\text{ is }(V_i,\eps^8,m)-\text{saturated}\right)>1-\eps^8,
\end{equation}
where $\omega=\eta^{*k}$. Hence, our task is reduced to proving (\ref{eq:ei3}).

Combining  Lemma \ref{lem:hoc} with the assumptions of Theorem \ref{thm:EntropyInverse},  $\delta=\eps^8/K$ and $\eps^{-1}, K\ll n$, we obtain:
\begin{equation}\nonumber
\begin{split}
\frac1n H\left(\omega*\pi_{\textbf{j}}\mu,\mathcal{L}_n\right)&\le \frac1n H\left(\eta*\pi_{\textbf{j}}\mu,\mathcal{L}_n\right) +k\left(\frac1n H\left(\eta*\pi_{\textbf{j}}\mu,\mathcal{L}_n\right)-\frac1n H\left(\pi_{\textbf{j}}\mu,\mathcal{L}_n\right) \right)+O(\frac kn)\\&\le
\frac1n H\left(\pi_{\textbf{j}}\mu,\mathcal{L}_n\right)+\delta +k\delta+\eps^8\le \alpha+O(\eps^8).
\end{split}
\end{equation}
Combining this with Lemma \ref{lem:eilow} and $R\ll\eps^{-1}\ll m\ll n$, we have
\begin{equation}\label{eq:A.3.6}
\mathbb{E}^{\omega}_{0\le i<n}\left(Q^{\omega}_m(i,y)\right)\le\alpha+O(\eps^{8}).
\end{equation}
By Theorem \ref{thm:ledrappier} and $1\ll \eps^{-1}\ll m$, we have
\begin{equation}\label{eq:A.3.8}
\nu^{\,\mathbb{Z}_+}\left(\textbf{z}\in\Sigma\::\:\left|\frac1m H\left(\,\pi_{\textbf{z}}\mu,\mathcal{L}_m\right) -\alpha\right|\le\eps^8\right)>1-\eps^8.
\end{equation}
Combining this with (\ref{ei22}) and    Lemma \ref{lem:trivialC},  we obtain:
\begin{equation}\label{eq:W.4.4}
Q^{\omega}_m(i,y) \ge\alpha-O(\eps^{8}) \quad\quad \text{for } \omega-\text{ a.e. } y.
\end{equation}
 Thus we have
$\mathbb{E}^{\omega}_{i}\left(Q^{\omega}_m(i,y)\right)\ge \alpha-O(\eps^{8})$
for each $0\le i<n$.
Combining this with (\ref{eq:A.3.6}) and Lemma \ref{lem:P}, we find a set  $I\subset\{0,\ldots,n-1\}$ such that $\frac{\# I}n\ge 1-O(\eps^4)$ and
\begin{equation}\label{eq:A.3.7}
  \mathbb{E}^{\omega}_{i}\left(Q^{\omega}_m(i,y)\right)\le\alpha+O(\eps^{4})\quad\quad\text{for }i\in I.
\end{equation}

Let $i\in I$.   Combining (\ref{eq:W.4.4}) with (\ref{eq:A.3.7}) and Lemma \ref{lem:P}, we obtain
\begin{equation}\label{eq:A.3.9}
\mathbb{P}^{\omega}_{i}\left(Q^{\omega}_m(i,y)\le\alpha+O(\eps^{2})\right)\ge 1-O(\eps^{2}) \quad\quad\text{for }i\in I.
\end{equation}

By Markov's inequality, (\ref{eq:A.3.10}) implies that there is a set $I'\subset\{0,\ldots,n-1\}$ such that
$$\mathbb{P}^{\omega}_{i}\left(\,S_{b^i}(\omega_{y,\,i})\text{ is }(V_i,\eps^8,m)-\text{saturated}\right)>1-O(\eps^{4})\quad\quad\text{for }i\in I',$$
with $\frac{\# I'}n\ge 1-O(\eps^4)$. Let $J=I\cap I'$.
 Choose $i\in J$ arbitrarily. Since $1\ll \eps^{-1}$, we can find a $y_i\in\mathbb{R}^d$ with
$Q^{\omega}_m(i,y_i)\le\alpha+O(\eps^{2})$ and $S_{b^i}(\omega_{y_i,\,i})$ is $(V_i,\eps^8,m)$-saturated.
Combining $Q^{\omega}_m(i,y_i)\le\alpha+O(\eps^{2})$  with (\ref{eq:A.3.8}), Lemma \ref{lem:trivialC} and  Lemma \ref{lem:P}, we obtain
$$\nu^{\,\mathbb{Z}_+}\left(\textbf{z}\in\Sigma\::\:\frac1m H\left(S_{b^i}(\omega_{\,y_i,\,i})\ast\pi_{\textbf{z}}\mu,\mathcal{L}_m\right) \le \alpha+O(\eps)\right)>1-O(\eps).$$
Combining this with (\ref{eq:A.3.8}), we have
$$\nu^{\,\mathbb{Z}_+}\left(\textbf{z}\in\Sigma\::\:\frac1m H\left(S_{b^i}(\omega_{\,y_i,\,i})\ast\pi_{\textbf{z}}\mu,\mathcal{L}_m\right) \le 
\frac1m H\left(\,\pi_{\textbf{z}}\mu,\mathcal{L}_m\right) 
+O(\eps)\right)>1-O(\eps).$$
Combining this with Lemma \ref{lem:saturated-trans} and the fact that  $S_{b^i}(\omega_{\,y_i,\,i})$ is $(V_i,\eps^8,m)$-saturated, then
\begin{equation}\nonumber
\nu^{\mathbb{Z}_+}\left(\textbf{z}\in\Sigma\::\: \pi_{\textbf{z}}\mu\text{ is }(V_i,O(\eps),m)-\text{saturated}\right)>1-O(\eps)\quad\quad\text{for }i\in J.
\end{equation}
 This implies that
  (\ref{eq:ei3}) holds, since  $\frac{\# J}n\ge 1-(1-\frac{\# I'}n)-(1-\frac{\# I'}n)\ge1-O(\eps^4)$.
\end{proof}

\section{Proof of Theorem A'}\label{sec:proveA'}
     In this section, let $\lambda\in(1/b,1)$ be fixed, and consider $\phi:\mathbb{R} \to \mathbb{R}^d$ as 
     a $\mathbb{Z}$-periodic analytic function satisfying the condition (H). Our approach to proving Theorem A' involves
     employing the method of contradiction. We start by introducing Theorem C, which plays a crucial role in
      deriving a contradiction. Then, we utilize Hochman's observation to infer pertinent properties, leading to 
      the completion of the theorem's proof.

\begin{thmC}
	 If $\alpha<\log_{\lambda^{-1}}b$,
	then there exists  $p_0=p_0(\alpha)>0$  with the following properties:
	for any $\delta>0$ and $n\ge N(\alpha,p_0,\delta)$, there exists
	a measures $\eta\in\mathscr{P}(\mathbb{R}^d)$  and $\textbf{z}\in\Sigma$ such that
	  $\text{dim}\left(\text{supp}(\eta)\right)\le R$  and:
	\begin{enumerate}
	\item [(C.1)]	 $\frac1{n}H\left(\eta,\mathcal{L}_n\right)\ge p_0$,
	\item [(C.2)] $\frac1{n}H\left(\eta*\pi_{\textbf{z}}\mu,\mathcal{L}_n\right)\le
	\alpha+\delta $,
	\item [(C.3)] $\left| \frac1{n}H\left(\pi_{\textbf{z}}\mu,\mathcal{L}_n\right)-\alpha\right|<\delta$,
	\end{enumerate}
where $R=R(\phi,b,\lambda)>0$ is a constant.
\end{thmC}

Most of the proofs of Theorem C are not new, and they closely resemble the details presented in \cite{ren2021dichotomy}. Therefore, for readers who are not familiar with \cite{ren2021dichotomy}, we offer a brief proof in the appendix.

\subsection{Saturated properties of $\pi_{\textbf{j}}\mu$}
The subsequent lemma is already included in the proof of \cite[Lemma 3.16]{Hochman2015}, but we provide the details here for completeness.
\begin{lemma}\label{lem:Ho151}
	Let $R \geq 1$,  let $V<\mathbb{R}^d$ be a linear space. Consider a measure $\omega \in \mathscr{P}(\mathbb{R}^d)$ such that $\text{diam}\left(\text{supp}(\omega)\right) \leq R$. Take integers $m \geq t \geq 1$. Then, the following  holds:
	\begin{equation}\nonumber
		\mathbb{E}_{0\le i<m}^{\omega }\left(\frac1{t}H\left(\omega_{y,\,i},\mathcal{L}_{t+i}\mid\mathcal{L}_{t+i}^{V^{\bot}}\right)\right)\ge
	\frac1{m}H\left(\omega,\mathcal{L}_m\mid\mathcal{L}_m^{V^{\bot}}\right)-O\left(\frac {tR}m+\frac1t\right).
	\end{equation}
\end{lemma}

\begin{proof}
 Using Lemma \ref{lem:minus} with $\mathcal{L}_{i}^{V^{\bot}}=\pi_{V^{\bot}}^{-1}\mathcal{L}_{i}$  and  $(\pi_{V^{\bot}}\omega)_{y,\,i}=(\pi_{V^{\bot}}\omega)_{\mathcal{L}_{i} (y)}$, we get:
\begin{equation}\label{pa2}
	\frac1{m}H\left(\omega,\mathcal{L}^{V^{\bot}}_m\right)=
	\frac1{m}H\left(\pi_{V^{\bot}}\omega,\mathcal{L}_m\right)
	=
	\mathbb{E}_{0\le i<m}^{\pi_{V^{\bot}}\omega }\left(\frac1{t}H\left((\pi_{V^{\bot}}\omega)_{\mathcal{L}_{i} (y)},\mathcal{L}_{t+i}\right)\right)+  O\left(\frac {t R}m\right).
	\end{equation}
By Lemma \ref{lem:minus} and (\ref{eq:CE}), we have:
\begin{equation}\nonumber
\begin{split}
		&\frac1{m}H\left(\omega,\mathcal{L}_m\right)=\mathbb{E}_{0\le i<m}^{\omega }\left(\frac1{t}H\left(\omega_{\mathcal{L}_{i} (y)},\mathcal{L}_{t+i}\right)\right)+ O\left(\frac {t R}m\right)
		\\&=\mathbb{E}_{0\le i<m}^{\omega }\left(\frac1{t}H\left(\omega_{\mathcal{L}_{i} (y)},\mathcal{L}_{t+i}\mid\mathcal{L}_{t+i}^{V^{\bot}}\right)\right)
		+\mathbb{E}_{0\le i<m}^{\omega }\left(\frac1{t}H\left(\omega_{\mathcal{L}_{i} (y)},\mathcal{L}_{t+i}^{V^{\bot}}\right)\right)+ O\left(\frac {t R}m+\frac1t\right).
\end{split}
\end{equation}
Combining this with (\ref{pa2}), we only need to prove for each $0\le i<m$:
\begin{equation}\label{eq:B1.1}
\mathbb{E}_{ i}^{\pi_{V^{\bot}}\omega }\left(\frac1{t}H\left((\pi_{V^{\bot}}\omega)_{\mathcal{L}_{i} (y)},\mathcal{L}_{t+i}\right)\right)\ge \mathbb{E}_{i}^{\omega }\left(\frac1{t}H\left(\omega_{\mathcal{L}_{i} (y)},\mathcal{L}_{t+i}^{V^{\bot}}\right)\right)+O(\frac1t).
\end{equation}

Easy to confirm:  $ \left(\pi_{V^{\bot}}\omega\right)_{\mathcal{L}_{i} (y)}= \pi_{V^{\bot}}(\omega_{\pi_{V^{\bot}}^{-1}\mathcal{L}_{i} (y)})=\pi_{V^{\bot}}(\omega_{\mathcal{L}^{V^{\bot}}_{i}(y) })$. Consequently,
\begin{equation}\nonumber
\begin{split}
&\mathbb{E}_{ i}^{\pi_{V^{\bot}}\omega}\left(\frac1{t}H\left((\pi_{V^{\bot}}\omega)_{\mathcal{L}_{i} (y)},\mathcal{L}_{t+i}\right)\right)=
\mathbb{E}_{ i}^{\omega}\left(\frac1{t}H\left(\omega_{\mathcal{L}^{V^{\bot}}_{i}(y)},\mathcal{L}^{V^{\bot} }_{t+i}\right)\right)
\\&\ge
\mathbb{E}_{i}^{\omega }\left(\frac1{t}H\left(\omega_{\mathcal{L}_{i}^{ V^{\bot}}\vee\mathcal{L}_{i}^{ V}\vee\mathcal{L}_{i} (y)},\mathcal{L}^{V^{\bot}}_{t+i}\right)\right)
\end{split}
\end{equation}
by cancavity. Combining this with  the convexity of entropy,  (\ref{eq:B1.1}) holds.
\end{proof}

We deduce the following application from  the observation in \cite[Theorem 6.11]{Hochman2015}: 
\begin{lemma}\label{lem:beta}
	For any linear space $V<\mathbb{R}^d$, we have
	$$\alpha-\alpha_{V^{\bot}}\ge\limsup_{t\to\infty}\limsup_{m\to\infty}	\int\mathbb{E}_{0\le i<m}^{\pi_{\textbf{j}}\mu }\left(\frac1{t}H\left((\pi_{\textbf{j}}\mu)_{y,\,i},\mathcal{L}_{t+i}\mid\mathcal{L}_{t+i}^{V^{\bot}}\right)\right)d\nu^{\,\mathbb{Z}_+}(\textbf{j}).$$
\end{lemma}
 We note that $\alpha_{V^{\bot}}=\alpha_{V^{\bot}}(\phi)$ is defined in Corollary \ref{cor:pro}.

\begin{proof}
Using  Theorem \ref{thm:ledrappier} and Lemma \ref{lem:minus}, we obtain:
\begin{equation}\nonumber
	\begin{split}
	\alpha&=\lim_{m\to\infty}	\int\frac1{m}H\left(\pi_{\textbf{j}}\mu,\mathcal{L}_m\right)d\nu^{\,\mathbb{Z}_+}(\textbf{j})=\lim_{t\to\infty}\lim_{m\to\infty}	\int\left(\frac1{m}H\left(\pi_{\textbf{j}}\mu,\mathcal{L}_m\right)+O\left(\frac tm+\frac1t\right)\right)d\nu^{\,\mathbb{Z}_+}(\textbf{j})\\&=
	\lim_{t\to\infty}\lim_{m\to\infty}	\int\mathbb{E}_{0\le i<m}^{\pi_{\textbf{j}}\mu }\left(\frac1{t}H\left((\pi_{\textbf{j}}\mu)_{y,\,i},\mathcal{L}_{t+i}\right)\right)d\nu^{\,\mathbb{Z}_+}(\textbf{j}).
	\end{split}
	\end{equation}
Thus, we aim to prove:
\begin{equation}\label{eq:B.1.2}
 \liminf_{t\to\infty}\liminf_{m\to\infty}	\int\mathbb{E}_{0\le i<m}^{\pi_{\textbf{j}}\mu }\left(\frac1{t}H\left((\pi_{\textbf{j}}\mu)_{y,\,i},\mathcal{L}_{t+i}^{V^{\bot}}\right)\right)d\nu^{\,\mathbb{Z}_+}(\textbf{j})\ge\alpha_{V^{\bot}}.
\end{equation}
Let   $1\ll t\ll m$.
By the definition of conditional entropy, we have
$$ \mathbb{E}^{\pi_{\textbf{j}}\mu }_{ i}\left(H\left((\pi_{\textbf{j}}\mu)_{y,\,i},\mathcal{L}_{t+i}^{V^{\bot}}\right)\right)= H\left(\pi_{\textbf{j}}\mu,\mathcal{L}_{t+i}^{V^{\bot}}\mid \mathcal{L}_{i}\right)\quad\quad\text{for each } i\in\mathbb{N},\,\textbf{j}\in\Sigma.$$
Combining this
 with  (\ref{eqn:self`affine'}) and $\pi_{\textbf{j}}g_{\textbf{a}}(x,y)=\lambda^{|\textbf{a}|}\pi_{{\textbf{a}}^*\textbf{j}}(x,y)+ \pi_{\textbf{j}}g_{\textbf{a}}(0,0) $, we get:
\begin{equation}\nonumber
	\begin{split}
	&\int\mathbb{E}_{0\le i<m}^{\pi_{\textbf{j}}\mu }\left(\frac1{t}H\left((\pi_{\textbf{j}}\mu)_{y,\,i},\mathcal{L}_{t+i}^{V^{\bot}}\right)\right)d\nu^{\,\mathbb{Z}_+}(\textbf{j})
	=\int\frac1m\sum_{i=0}^{m-1}\frac1{t}H\left(\pi_{\textbf{j}}\mu,\mathcal{L}_{t+i}^{V^{\bot}}\mid \mathcal{L}_{i}\right)d\nu^{\,\mathbb{Z}_+}(\textbf{j})
	\\&\ge
	\frac1m\sum_{i=0}^{m-1}\int\frac1{b^{\hat{i}}}\sum_{\textbf{a}\in\Lambda^{\hat{i}}}\frac1{t}H\left(\pi_{\textbf{j}}g_{\textbf{a}}\mu,\mathcal{L}_{t+i}^{V^{\bot}}\mid\mathcal{L}_{i}\right)d\nu^{\,\mathbb{Z}_+}(\textbf{j})\\&=
\frac1m\sum_{i=0}^{m-1}\int\frac1{b^{\hat{i}}}\sum_{\textbf{a}\in\Lambda^{\hat{i}}}\frac1tH\left(\pi_{\textbf{a}^*\textbf{j}}\mu,\mathcal{L}_{t}^{V^{\bot}}\right) d\nu^{\,\mathbb{Z}_+}(\textbf{j})+O(\frac1t).
	\end{split}
	\end{equation}
Combining this with $\sup_{\textbf{j}'\in\Sigma}\parallel \pi_{\textbf{a}^*\textbf{j}}-\pi_{\textbf{a}^*\textbf{j}'}\parallel_{\infty} = O(b^{-t})$ for $\gamma^{\hat{i}}\le b^{-t}$, we conclude:
$$\int\mathbb{E}_{0\le i<m}^{\pi_{\textbf{j}}\mu }\left(\frac1{t}H\left((\pi_{\textbf{j}}\mu)_{y,\,i},\mathcal{L}_{t+i}^{V^{\bot}}\right)\right)d\nu^{\,\mathbb{Z}_+}(\textbf{j})\ge \int\frac1{t}H\left(\pi_{\textbf{j}}\mu,\mathcal{L}^{V^{\bot}}_{t}\right)d\nu^{\,\mathbb{Z}_+}(\textbf{j})+O\left(\frac tm+\frac1t\right).$$
Thus, (\ref{eq:B.1.2}) holds, as $\lim\limits_{t\to\infty} \int\frac1{t}H\left(\pi_{\textbf{j}}\mu,\mathcal{L}^{V^{\bot}}_{t}\right)d\nu^{\,\mathbb{Z}_+}(\textbf{j})=\alpha_{V^{\bot}} $ by  Corollary \ref{cor:pro}.
\end{proof}

Let $Gr_{\ell}(d)$ denote the Grassmannian of $\ell$-dimensional linear spaces of $\mathbb{R}^d$, where $0 \leq \ell \leq d$. The metric is defined as:
$$d(V,V')=\sup_{a\in\mathbb{R}^d,\,|a|\le1}\left|\pi_V(a)-\pi_{V'}(a)\right|,\quad\quad\text{for }V,V'\in Gr_{\ell}(d).$$
Additionally, note that $Gr_{\ell}(d)$, equipped with this metric, is compact.

\begin{lemma}\label{cor:pa}
	Given $V\in  Gr_{\ell}(d)$,  consider a sequence of integers $0<m_1<m_2<\ldots$ and subspaces $V_1,V_2,\ldots\in Gr_{\ell}(d)$ such that $V_k\to V$.
         Then we have:
	$$\alpha-\alpha_{V^{\bot}}\ge\limsup_{k\to\infty}\int\frac1{m_k}H\left(\pi_{\textbf{j} }\mu,\mathcal{L}_{m_k}\mid\mathcal{L}_{m_k}^{V_k^{\bot}}\right)d\nu^{\,\mathbb{Z}_+}(\textbf{j}).$$
\end{lemma}

\begin{proof}
Given $\eps > 0$, let integers $t,k \ge1$ be such that $\eps^{-1} \ll t \ll k$. Therefore, by Lemma \ref{lem:beta}, we have:
\begin{equation}\label{pa7}
	\alpha-\alpha_{V^{\bot}}+\eps\ge \int\mathbb{E}_{0\le i<m_k}^{\pi_{\textbf{j}}\mu }\left(\frac1{t}H\left((\pi_{\textbf{j}}\mu)_{y,\,i},\mathcal{L}_{t+i}\mid\mathcal{L}_{t+i}^{V^{\bot}}\right)\right)d\nu^{\,\mathbb{Z}_+}(\textbf{j}).
\end{equation}
   Now, combining $V^{\bot}_k\to V^{\bot}$ with $\eps^{-1}\ll t\ll k$, (\ref{eq:ero-minus}), and Lemma \ref{lem:Ho151}, we get:
\begin{equation}\nonumber
	\begin{split}
	&\int\mathbb{E}_{0\le i<m_k}^{\pi_{\textbf{j}}\mu }\left(\frac1{t}H\left((\pi_{\textbf{j}}\mu)_{y,\,i},\mathcal{L}_{t+i}\mid\mathcal{L}_{t+i}^{V^{\bot}}\right)\right)d\nu^{\,\mathbb{Z}_+}(\textbf{j})\\&=
	\int\mathbb{E}_{0\le i<m_k}^{\pi_{\textbf{j}}\mu }\left(\frac1{t}H\left((\pi_{\textbf{j}}\mu)_{y,\,i},\mathcal{L}_t\right)-
	\frac1{t}H\left(    \pi_{V^{\bot}} \left(b^i(\pi_{\textbf{j}}\mu)_{y,\,i}\right),\mathcal{L}_t\right)
	\right)d\nu^{\,\mathbb{Z}_+}(\textbf{j})+O(\frac1t)\\&=
        \int\mathbb{E}_{0\le i<m_k}^{\pi_{\textbf{j}}\mu }\left(\frac1{t}H\left((\pi_{\textbf{j}}\mu)_{y,\,i},\mathcal{L}_t\right)-
	\frac1{t}H\left(    \pi_{V_k^{\bot}} \left(b^i(\pi_{\textbf{j}}\mu)_{y,\,i}\right)    ,\mathcal{L}_t\right)
	\right)d\nu^{\,\mathbb{Z}_+}(\textbf{j})+O\left(\frac1t+\frac t{m_k}\right)
        \\&=
	 \int\mathbb{E}_{0\le i<m_k}^{\pi_{\textbf{j}}\mu }\left(\frac1{t}H\left((\pi_{\textbf{j}}\mu)_{y,\,i},\mathcal{L}_{t+i}\mid\mathcal{L}_{t+i}^{V_k^{\bot}}\right)\right)d\nu^{\,\mathbb{Z}_+}(\textbf{j})+O\left(\frac1t+\frac t{m_k}\right)
          \\&\ge\int
	\frac1{m_k}H\left(\pi_{\textbf{j} }\mu,\mathcal{L}_{m_k}\mid\mathcal{L}_{m_k}^{V_k^{\bot}}\right)d\nu^{\,\mathbb{Z}_+}(\textbf{j})-\eps.
	\end{split}
	\end{equation}
Combining this with (\ref{pa7}), then our claim holds.
\end{proof}

\subsection{Proof of Theorem A'} 
The following lemma is a well-known fact from \cite{hochman2014self}, and detailed proof can be found in \cite[Lemma 5.6]{ren2021dichotomy}.

\begin{lemma}\label{lem:Hochman3}
	 Let $p _0> 0$ and $R \geq 1$. For $0 < \eps < \eps_0(p_0, R)$ and $m \geq M(p_0, \eps, R)$, the following holds:
Consider $\eta \in \mathscr{P}(\mathbb{R}^d)$ such that $\text{diam}\left(\text{supp}(\eta)\right) \leq R$. If
	\begin{equation}\nonumber
	\mathbb{P}^{\eta}_{0\le i< m} \left(
	\begin{matrix}
	\eta_{y,\,i}\text{ is } (\{0\},\eps/b^i)-\text{concentrated}
	\end{matrix}
	\,\right) > 1-\eps,
	\end{equation}
	then
	$\frac1mH(\eta,\mathcal{L}_m)<p_0.$
\end{lemma}

\begin{proof}[Proof of Theorem A']
   %  When $d=1$, Theorem A' is proven by \cite[Theorem B']{ren2021dichotomy}. For $d>1$, using induction, we assume Theorem A' holds for dimensions up to $d-1$. To establish the case for dimension $d$, we assume $\alpha<\min\{d,\log_{\lambda^{-1}}b\}$.
 We only need to prove the following claim:

\begin{claim}\label{claim:1}
Either $\alpha(\phi)=\min\{d,\log_{\lambda^{-1}}b\}$, or there exists a linear space $V<\mathbb{R}^d$ such that $0<\text{dim}(V)<d$ and
$$\alpha(\phi)-\alpha_{V^{\bot}}(\phi)\ge \text{dim}(V).$$
\end{claim}

In particular, this claim implies that Theorem A' holds directly when $d=1$. For $d>1$, using induction, we assume Theorem A' holds for dimensions up to $d-1$. To establish the case for dimension $d$, we assume $\alpha(\phi)<\min\{d,\log_{\lambda^{-1}}b\}$. Let $V<\mathbb{R}^d$ be the linear space as stated in Claim \ref{claim:1}.

Considering $\pi_{V^{\bot}}\phi$ as a function with a range in $\mathbb{R}^{\text{dim}(V^{\bot})}$, satisfying condition (H) because $\phi$ satisfies it, we obtain $\alpha\left(\pi_{V^{\bot}}\phi \right)=\min\{\text{dim}(V^{\bot}),\log_{\lambda^{-1}}b\}$. Combining this with $\alpha_{V^{\bot}}(\phi)=\alpha\left(\pi_{V^{\bot}}\phi \right)$ and Claim \ref{claim:1}, we reach a contradiction, confirming Theorem A' for dimension $d$.

Now, let's prove Claim \ref{claim:1}. Suppose $\alpha(\phi)<\min\{d,\log_{\lambda^{-1}}b\}$.
Let $p_0 > 0$ and $R \geq 1$ be as specified in Theorem C. Let $\eps_k > 0$ be such that $R \ll \eps_k^{-1}$ and $\eps_k \to 0$. Then, we define a sequence of $\delta_k > 0$ and integers $m_k, n_k \geq 1$, for $k = 1, 2, \ldots$, such that $\eps_k^{-1} \ll m_k \ll \delta_k^{-1} \ll n_k$.
 Consequently, 
there exist measures $\eta_{k}\in\mathscr{P}(\mathbb{R}^d)$ with $\text{dim}\left(\text{supp}(\eta)\right)\le R$, and $\textbf{j}_k\in\Sigma$ satisfying (C.1), (C.2) and (C.3)  for $n_k$ and $\delta_k$ according to  Theorem C. Combining this with Theorem \ref{thm:EntropyInverse} and $R\ll \eps^{-1}_k\ll m_k\ll\delta_k^{-1}\ll n_k$,  we can find a sequence of linear spaces   $V_{0,\,k}, V_{1,\,k},\ldots, V_{n_k-1,\,k}<\mathbb{R}^d$ such that (\ref{eq:ei2}) and (\ref{eq:ei3}) hold for $\eps_k$, $m_k$ and $\eta_{k},\, \textbf{j}_k$.

When $k\geq 1$, Markov's inequality, as applied to equations (\ref{eq:ei2}) and (\ref{eq:ei3}), guarantees the existence of a set $I_k \subset \{0,1,\ldots,n_k-1\}$ such that $\frac{\# I_k}{n_k}\ge 1-O(\sqrt{\eps_k})$. Furthermore, for every $i \in I_k$, we have
\begin{equation}\label{pa3}
\mathbb{P}^{\eta_k}_{ i}\left((\eta_k)_{x,\,i}\text{ is } (V_{i,\,k},\eps_k/b^i)-\text{concentrated} \right)>1-O(\sqrt{\eps_k}),
\end{equation}
\begin{equation}\label{pa4}
\nu^{\,\mathbb{Z}_+}\left(\left\{\textbf{w}\in\Sigma\::\:\pi_{\textbf{w}}\mu\text{ is } (V_{i,\,k},\eps_k,m_k)-\text{saturated}\right\} \right)>1-O(\sqrt{\eps_k}).
\end{equation}
If there is a $i\in I_k$ such that $V_{i,\,k}=\mathbb{R}^d$, then according to (\ref{pa4}), we infer $\alpha \ge d - O(\eps_k)$ since $\eps_k^{-1} \ll m_k$ and (\ref{eq:LCT}). However, this contradicts the conditions $\alpha < d$ and $1 \ll \eps_k^{-1}$. Thus, we conclude that $\text{dim} V_{i,\,k} < d$ for every $i \in I_k$.
Furthermore, if $\text{dim} V_{i,\,k} = 0$ for any $i \in I_k$, then combining (\ref{pa3}) with $\frac{\# I_k}{n_k} \ge 1 - O(\sqrt{\eps_k})$ implies that (C.1) fails for $\eta_k$, as per Lemma \ref{lem:Hochman3}. Consequently,  $d>1$ and there exists an $i_k \in I_k$ such that $0<\text{dim} V_{i_k,\,k} < d$. Denote $V_k := V_{i_k,\,k}$. Subsequently, (\ref{pa4}) implies
\begin{equation}\label{pa5}
\int\frac1{m_k}H\left(\pi_{\textbf{j}}\mu,\mathcal{L}_{m_k}\mid\mathcal{L}_{m_k}^{V_k^{\bot}}\right)d\nu^{\,\mathbb{Z}_+}(\textbf{j})\ge\text{dim}V_k-O(\sqrt{\eps_k}).
\end{equation}

Utilizing compactness, we can assume the existence of an integer $0<\ell<d$ and  a linear subspace
$V\in Gr_{\ell}(d)$ such that $V_k\in  Gr_{\ell}(d)$ and $V_k\to V$. Combining this with (\ref{pa5}) and Lemma \ref{cor:pa}, we conclude that Claim \ref{claim:1} holds.
\end{proof}

\section{Proofs of Theorem B and Corollary \ref{cor2}}\label{sec:proveB}
In this section, we will first provide an equivalent characterization of  the condition (H$^*$)  using Fourier coefficients. Then, we will combine this with the theories of linear algebra and complex analysis to complete the proof of Theorem B, and finally conclude with the proof of Corollary \ref{cor2}.

%Throughout this section, we will assume that  $\phi:\mathbb{R}\to\mathbb{R}^d$ is a non-constant 
%	$\mathbb{Z}$-periodic $C^3$ function. Specifically, we denote  $\phi(x)=(\phi_1(x),\ldots,\phi_d(x))$.

\subsection{Fourier Analysis of  (H$^*$)}The following directly follows from \cite[Theorem 2.1]{gao2022}.
\begin{lemma}\label{cor:gs}
	Given $\lambda\in(1/b,1)$ and a $\mathbb{Z}$-periodic $C^1$ function $\hat{f}:\mathbb{R}\to\mathbb{R}$, 
	 the following statements are equivalent:
	\begin{enumerate}
		\item [(i)] $\hat{f}$ satisfies  the condition (H$^*$);
		\item [(ii)] There exist  $\textbf{i}\neq\textbf{j}\in\Sigma$ such that 
		$Y^{\hat{f}}_{\lambda,\,b}(x,\,\textbf{i})\equiv Y^{\hat{f}}_{\lambda,\,b}(x,\,\textbf{j});$
		\item [(iii)] there is a real $\mathbb{Z}$-periodic $C^1$ function $\psi$ such that
		$\hat{f}(x)=\lambda \psi(bx)- \psi(x)+Constant.$
	\end{enumerate}	
\end{lemma}
\begin{remark}\label{rem:C^k}
 (iii) implies $\psi(x)=\sum_{n=1}^{\infty}\frac1{\lambda^n}\left(\hat{f}(x/b^n)-\hat{f}(0)\right)+\psi(0)$
  through direct calculation. Consequently,  if $\hat{f}$ is $C^k$ for some $k\ge1$, then $\psi$ is also a $C^k$ function.
\end{remark}

\begin{proof}
 Define the function $h_{\textbf{i}}:\mathbb{R}\to\mathbb{R}$ for any $\textbf{i}\in\Sigma$ as:
\begin{equation}\label{def:h_i}
	h_{\textbf{i} }(x)=\sum_{n=1}^{\infty}\frac1{\lambda^n}\left[\,\hat{f}\left(\frac{x+i_1+i_2b\ldots+i_nb^{n-1}}{b^n }\right)- \hat{f}\left(\frac{i_1+i_2b\ldots+i_nb^{n-1}}{b^n }\right)\,\right].
\end{equation}
Hence, (\ref{def:Y})  implies  $h'_{\textbf{i} }(x)\equiv-Y^{\hat{f}}_{\lambda,\,b}(x,\textbf{i})$. With $h_{\textbf{i}}(0)=0$, we have:
$$  h_{\textbf{i} }(x)\equiv h_{\textbf{j}}(x)\quad\quad\text{and} \quad \quad  Y^{\hat{f}}_{\lambda,\,b}(x,\textbf{i})\equiv Y^{\hat{f}}_{\lambda,\,b}(x,\textbf{j})$$
 equivalent for all  $\textbf{i},\textbf{j}\in\Sigma$. Combining this with \cite[Theorem 2.1]{gao2022}, our claim holds.
\end{proof}

For any $\mathbb{Z}$-periodic continuous function $\phi:\mathbb{R}\to\mathbb{R}$ and $m\in\mathbb{Z}$, we define the $m$-th Fourier coefficient of $\phi$ as follows:
\begin{equation}
a_m(\phi)=\int_{0}^{1}\phi(x)e^{-2\pi i\,mx }dx.
\end{equation}

\begin{proposition}\label{lem:Fourier}
	 Suppose $\lambda\in(1/b,1)$ and $\phi:\mathbb{R}\to\mathbb{R}$ is a $\mathbb{Z}$-periodic function with $C^3$ continuity. Then $\phi$ satisfies the condition (H$^*$) if and only if:
    \begin{equation}\label{eq:FE}
    \sum_{k=0}^{\infty} a_{tb^k}(\phi)\cdot\lambda^{-k}=0,\quad\quad\text{for each } t\in\mathbb{Z}\text{ with }   b \nmid   t.
    \end{equation}
\end{proposition}

\begin{proof}
If $\phi$ satisfies the condition (H$^*$), then by Lemma \ref{cor:gs} and Remark \ref{rem:C^k}, there exists a $\mathbb{Z}$-periodic $C^3$ function $\psi:\mathbb{R}\to\mathbb{R}$ such that $\phi(x)=\lambda \psi(bx)- \psi(x)+Constant$. Comparing Fourier coefficients after expanding both sides into Fourier series, we get:
\begin{equation}
a_m(\phi)=
\begin{dcases}\nonumber
 a_{m/b}(\psi) \cdot\lambda -a_m(\psi), &  b\mid m;\\
-a_m(\psi),                  & b\nmid m.
\end{dcases}
\end{equation}
Hence, for any integer $t$ such that $b\nmid t$ and $n\in\mathbb{N}$, we have:
    $$-a_{tb^n}(\psi)=a_{tb^n}(\phi)+ a_{tb^{n-1}}(\phi) \cdot\lambda+\ldots+ a_{t}(\phi)\cdot\lambda^n,$$
    This implies $ - a_{tb^n}(\psi)\cdot \lambda^{-n}= \sum_{k=0}^{n} a_{tb^k}(\phi)\cdot\lambda^{-k}$.
   By combining this with $|a_{tb^n}(\psi)|=O(\frac1{tb^n})$ and $\frac1{b\lambda}\in(1/b,1)$,  (\ref{eq:FE}) holds.

If  (\ref{eq:FE})  holds for every integer $t$ such that $b\nmid t$ and $n\in\mathbb{N}$,  we define
    \begin{equation}\label{def:dc}
    d_{tb^n}=-\left(a_{tb^n}(\phi)+ a_{tb^{n-1}}(\phi)\cdot\lambda+\ldots+ a_{t}(\phi)\cdot\lambda^n\right).
    \end{equation}
This implies $d_{tb^n}=-\lambda^{n}\sum_{k=n+1}^{\infty}a_{tb^k}(\phi)\cdot\lambda^{-k}$. Given $|a_{tb^k}(\phi)|=O(\frac1{t^3b^{3k}})$, we conclude $|d_{tb^k}|=O(\frac1{t^3b^{3k}})$. Combining this with (\ref{def:dc}), we obtain that the $\mathbb{Z}$-periodic function $\psi(x):=\sum_{n\in\mathbb{Z}\setminus\{0\}}d_{n}\cdot e^{2\pi i\,nx}$ is a real $C^1$ function, and $\phi(x)=\lambda \psi(bx)- \psi(x)+Constant$ through Fourier coefficient analysis. Thus, $\phi(x)$ satisfies the condition (H$^*$) by Lemma \ref{cor:gs}.
\end{proof}

\subsection{Proof of Theorem B} In this subsection, we fix a $\mathbb{Z}$-periodic $C^3$ function $\phi:\mathbb{R}\to\mathbb{R}^d$.  Specifically,  we denote $\phi(x)=(\phi_1(x),\phi_2(x),\ldots,\phi_d(x))$. We list all positive integers $t$ (such that $b \nmid t$) in ascending order as $t_1 < t_2 < \ldots$.

For any $1\le j\le d$ and $m\in\mathbb{Z}+$, we define functions $R_{j,\,m},\,I_{j,\,m}:[1/b,1]\to\mathbb{R}$ as follows:
\begin{equation}\label{def:R-I}
	R_{j,\,m}(\lambda)=\sum_{n=0}^{\infty}  \text{Re}\left\{a_{t_mb^n}(\phi_j)\right\}\cdot\lambda^{-n}\quad\quad
 I_{j,\,m}(\lambda)=\sum_{n=0}^{\infty} \text{Im}\left\{a_{t_mb^n}(\phi_j)\right\}\cdot\lambda^{-n},
\end{equation}
	where $\text{Re}\{\cdot\}$ denotes the real part and $\text{Im}\{\cdot\}$ denotes the imaginary part. 
Because $\phi$ is a $C^3$ function, we have $|a_{t_mb^n}(\phi_j)| = O\left(\frac{1}{t_m^3b^{3n}}\right)$. Consequently, $R_{j,\,m}$ and $I_{j,\,m}$ are real analytic.

Given $m\in\mathbb{Z}_+$ and $\lambda\in[1/b,1]$, let's define the matrix $A_{\lambda,\,m}\in M_{d,\,2m}(\mathbb{R})$ as follows:
\begin{equation}\label{def:matrixA}
A_{\lambda,\,m}=
\begin{pmatrix}
R_{1,\,1}(\lambda)&\ldots& R_{1,\,m}(\lambda)&I_{1,\,1}(\lambda)&\ldots& I_{1,\,m}(\lambda)\\
R_{2,\,1}(\lambda)&\ldots& R_{2,\,m}(\lambda)&I_{2,\,1}(\lambda)&\ldots& I_{2,\,m}(\lambda)\\
\vdots&&\vdots&  \vdots&&\vdots\\
R_{d,\,1}(\lambda)&\ldots& R_{d,\,m}(\lambda)&I_{d,\,1}(\lambda)&\ldots& I_{d,\,m}(\lambda)
\end{pmatrix}.
\end{equation}
The kernel space of $A_{\lambda,\,m}$ is defined as:
\begin{equation}\label{def:kerA}
\text{Ker } A_{\lambda,\,m}=\left\{\,y=(y_1,\ldots,y_d)\in\mathbb{R}^d:\,yA_{\lambda,\,m}=0\,\right\}.
\end{equation}
We observe that $\text{Ker } A_{\lambda,\,m+1}\subset\text{Ker } A_{\lambda,\,m}.$
Recall  that $	q'(\phi,\lambda,b)$ is defined by (\ref{def:q'}).

\begin{lemma}\label{lem:dimq}
	For any $\lambda\in(1/b,1)$,  we have
	$q'(\phi,\lambda,b)=\text{dim}\left(\cap_{m=1}^{\infty} \text{Ker } A_{\lambda,\,m} \right).$
\end{lemma}

\begin{proof}
Let  $y\in \mathbb{R}^d$  such that $\pi_{\mathbb{R} y} \phi$ satisfies the condition (H$^*$). Considering $\pi_{\mathbb{R} y} \phi$ as a function with range in $\mathbb{R}$, we have:
$$ \sum_{n=0}^{\infty} a_{t_mb^n} \left(\sum_{j=1}^dy_j\phi_j\right)\cdot\lambda^{-n}=\sum_{j=1}^dy_j  \sum_{n=0}^{\infty} a_{t_mb^n} (\phi_j)\cdot\lambda^{-n}=0\quad\quad\forall\,m\in\mathbb{Z}_+$$
by Proposition \ref{lem:Fourier}. Combining this with (\ref{def:R-I}), we get:
$$\sum_{j=1}^dy_j\cdot R_{j,\,m}(\lambda)=0\quad\quad\text{and}\quad\quad \sum_{j=1}^dy_j\cdot I_{j,\,m}(\lambda)=0, \quad\quad\forall\,m\in\mathbb{Z}_+.$$
Hence,  $y\in \cap_{m=1}^{\infty} \text{Ker } A_{\lambda,\,m}$ by (\ref{def:kerA}), implying  $q'\le\text{dim}\left( \cap_{m=1}^{\infty}  \text{Ker } A_{\lambda,\,m} \right)$ by  the arbitrariness of $y$ and the definition of $q'$.
By Lemma \ref{lem:Fourier} and similar methods, $\pi_{  \cap_{m=1}^{\infty} \text{Ker } A_{\lambda,\,m}}(\phi)$ satisfies the condition (H$^*$), implying
$\text{dim}\left( \cap_{m=1}^{\infty}  \text{Ker } A_{\lambda,\,m} \right)\le q'$  by the maximality of $q'$.
\end{proof}

Let \( m, n \in \mathbb{Z}_+ \) and \( B = (b_{i,\,j})_{n\times m} \in M_{n,\,m}(\mathbb{R}) \) be a matrix. For any \( 1 \leq i_1 < \ldots < i_k \leq n \) and \( 1 \leq j_1 < \ldots < j_{k'} \leq m \), define the submatrix:
\begin{equation}\label{def:subm}
B\left(\{i_1,\ldots,i_k\}, \{j_1,\ldots,j_{k'}\}\right) =
\begin{pmatrix}
b_{i_1,\,j_1} & \ldots & b_{i_1,\,j_{k'}}\\
\vdots & \ddots & \vdots\\
b_{i_k,\,j_1} & \ldots & b_{i_k,\,j_{k'}}
\end{pmatrix}.
\end{equation}
For any \(1\le k \leq \min\{m,n\} \), let
\begin{equation}\label{def:L_k}
L_k(B) = \sum_{I\subset\{1,\ldots,n\},\,  J\subset\{1,\ldots,m\},\,\#I=\# J=k}
\left( \text{det} \, B(I,J) \right)^2,
\end{equation}
where \( \text{det}(\cdot) \) denotes the determinant. 
The following is a basic fact in linear algebra.

\begin{lemma}\label{lem:la}
For any matrix \( B \in M_{n,\,m}(\mathbb{R}) \), we have:
\begin{equation}\nonumber
n - \text{dim} \left\{ y \in \mathbb{R}^n \::\: yB = 0 \right\} = \max_{1\le k \leq \min\{m,n\} } \left\{ k \:: \: L_k(B) \neq 0 \right\}.
\end{equation}
\end{lemma}
For any $1 \leq j \leq d$, define the function \(D_j: [1/b, 1] \rightarrow \mathbb{R}^{\mathbb{Z}_+}\) as follows:
\begin{equation}\label{def:D_j}
D_j(\lambda) = \left( L_j\left( A_1(\lambda) \right),\, L_j\left( A_2(\lambda) \right), \ldots \right),
\end{equation}
where \(A_m(\lambda)\) is defined by  (\ref{def:matrixA}) (set, $L_j\left( A_m(\lambda) \right)=0$ if $j>m$).
Recall that $p'(\phi, b)$ is defined by (\ref{def:p'}).

\begin{lemma}\label{lem:Caculate}
		 For any   $\lambda\in(1/b,1)$, we have
		\begin{equation}\label{eq:dimq11}
		d-q'(\phi,\lambda,b)=\max_{1\le j\le d}\left\{j\::\:D_j(\lambda)\neq0\right\}.
		\end{equation}	
       Furthermore, we get:
       \begin{equation}\label{eq:dimq22}
		d-p'(\phi,b)=\max_{1\le j\le d,\,\lambda\in(1/b,1)}\left\{j\::\:D_j(\lambda)\neq0\right\}.
		\end{equation}
\end{lemma}

\begin{proof}
By (\ref{def:p'}), we only need to prove (\ref{eq:dimq11}). Since $\text{ker }A_{\lambda,\,m+1}\subset\text{ker }A_{\lambda,\,m}$ holds for any $m\ge1$, there exists $m_0\ge d$ such that 
$\text{ker }A_{\lambda,\,m_0}=\cap_{m\ge1}\text{ker }A_{\lambda,\,m}$. Combining this with Lemma \ref{lem:dimq}, we have
\[ q'(\phi,\lambda,b)=\text{dim }\left(\text{ker }A_{\lambda,\,m}\right) \quad \text{for any } m \geq m_0. \]
Combining this with Lemma \ref{lem:la}, we get
\begin{equation}\label{eq:C.2.1}
    d - q'(\phi,\lambda,b) = \max_{1\le j\le d}\left\{j \::\: L_j\left(A_{\lambda,\,m}\right) \neq 0 \right\}, \quad \text{for } m \geq m_0.
\end{equation}
This implies that 
\begin{equation}\label{eq:dimq4}
    L_j\left(A_{\lambda,\,m}\right) = 0 \quad \forall m \geq m_0, \, j > d - q'(\phi,\lambda,b).
\end{equation}

Since $\text{ker }A_{\lambda,\,m_0}\subset\text{ker }A_{\lambda,\,m}$ holds for any $m \leq m_0$, we have
\[ L_j\left(A_{\lambda,\,m}\right) = 0 \quad \forall m \leq m_0, \, j > d - q'(\phi,\lambda,b) \]
by Lemma \ref{lem:la}.
 Combining this with (\ref{eq:dimq4}) and (\ref{eq:C.2.1}), we conclude that (\ref{eq:dimq11}) holds.
\end{proof}

\begin{proof}[Proof of Theorem B]
	Since $\phi$ is non-constant, there exist $1 \leq k\leq d$ and $m_0 \in \mathbb{Z}_+$, $\ell \in \mathbb{N}$ such that $a_{t_{m_0}b^{\ell}}(\phi_k) \neq 0$. Thus $\sum_{n=0}^{\infty}  a_{t_{m_0}b^n}(\phi_k)\cdot\lambda^{-n} \not\equiv 0$. Consequently, there exists $\lambda_0 \in (1/b,1)$ such that $D_1(\lambda_0) \neq 0$, implying $p'(\phi,b) < d$ by (\ref{eq:dimq22}).

By (\ref{eq:dimq22}) and (\ref{def:D_j}), there exists $m_1 \in \mathbb{Z}_+$ and $\lambda_1 \in (1/b,1)$ such that $L_{d-p'}(A_{\lambda_1,\,m_1}) \neq 0$. Since the function $F:[1/b,1] \to \mathbb{R}$ defined by $F(\lambda) = L_{d-p'}(A_{\lambda,\,m_1})$ is analytic, there are at most finitely many $\lambda \in [1/b,1]$ such that $D_{d-p'}(\lambda) = 0$ by (\ref{def:D_j}). Combining this with (\ref{eq:dimq11}) and (\ref{eq:dimq22}), the number of elements $\lambda \in [1/b,1]$ such that $d-q'(\phi,\lambda,b) < d-p'$ is at most finite.
\end{proof}

\begin{proof}[Proof of Corollary \ref{cor2}]
	We no longer fix the function $\phi$, and denote notations with a superscript $\phi$ to indicate their dependence on the choice of $\phi$. By Lemma \ref{lem:Caculate}, we have
\begin{equation}\label{eq:K}
K=\bigcup_{\lambda\in(1/b,1)}\left\{\phi\in X \::\: D^{\phi}_d(\lambda)\neq0\right\}.
\end{equation}
To prove that $K$ is open, we need to establish the following claim.
\begin{claim}\label{cla:C.4.1}
For any $\lambda\in(1/b,1)$, $1\leq j \leq d$, and $m\geq 1$, the linear functionals $R^{\phi}_{j,\,m}(\lambda),\,I^{\phi}_{j,\,m}(\lambda):X\to\mathbb{R}$ are continuous.
\end{claim}
Specifically, combining this with (\ref{def:L_k}) and (\ref{def:D_j}), we conclude that $\left\{\phi\in X\,:\, D^{\phi}_d(\lambda)\neq0\right\}$ is open. Consequently, $K$ is open as well by (\ref{eq:K}). Now, we prove the claim. Since 
\[
b^n\cdot \left|\,a_{t_mb^n}(\phi_j)\,\right| \leq \parallel \phi'_j\parallel_{\infty}  \leq  \parallel \phi \parallel_{r},\quad\quad\text{for }n\in\mathbb{N},
\]
we have
\[
\left|\sum_{n=0}^{\infty}a_{t_mb^n}(\phi_j)\cdot\lambda^{-n}\right|\leq  \frac{b\lambda}{b\lambda-1}\parallel \phi \parallel_{r}.
\]
Combining this with (\ref{def:R-I}),  then $R^{\phi}_{j,\,m}(\lambda),\,I^{\phi}_{j,\,m}(\lambda)$ are bounded $\mathbb{R}$-linear functionals.

Let $\lambda\in(1/b,1)$, and let $\{e_j\}_{1\leq j\leq d}$ be the standard basis of $\mathbb{R}^d$. Set $[d]=\{1,\ldots,d\}$. For any $\phi\in X$, we define the function $G:\mathbb{R}\to\mathbb{R}$ as
\begin{equation}\nonumber
    G(r) = \det\left( A^{\phi(x)-2r\sum_{j=1}^d \cos(2\pi t_jx)e_j}_{\lambda,\,d}([d],[d]) \right).
\end{equation}
Since $G(r)=\det\left(A^{\phi}_{\lambda,\,d}([d],[d]) - r I\right)$ is a polynomial function (with $I$ denoting the identity map), for any $\eps>0$, there exists $|r|<\eps$ such that $G(r)\neq0$. Using (\ref{eq:K}) and the defintion of $D_d(\lambda)$, this implies that the function $\phi(x)-2r\sum_{j=1}^d \cos(2\pi t_jx)e_j\in K$.  We conclude that $K$ is a dense set.
\end{proof}

\section{Appendix: Proof of Theorem C}
  In this appendix, we fix $\lambda\in\left(1/b,1\right)$ and a $\mathbb{Z}$-periodic analytic function $\phi:\mathbb{R}\to\mathbb{R}^d$ satisfying the condition (H), where $\phi(x)=\left(\phi_1(x),\phi_2(x),\ldots,\phi_d(x)\right)$. Without loss of generality, we set $\phi(0)=0$ for convenience. 
We assume  $\alpha<\log_{\lambda^{-1}}b$.
We extend results from $d=1$ to $d\geq 1$ as presented in \cite{ren2021dichotomy}. Since many of these results can be proven with similar details, we provide statements without proofs. These results are then used to prove Theorem C.

\subsection{Transversality}\label{sec:separation}
The following statement directly follows from \cite[Theorem 5.1]{ren2021dichotomy}.

\begin{theorem}\label{thm:estimate}
    There exist positive integers $\ell_0$, $Q_0$, and a constant $\rho_0 > 0$ such that for each $1 \leq j \leq d$ and for $\textbf{u}, \textbf{v} \in \Sigma$ with $u_n \neq v_n$, the following holds: 
    \begin{equation}\label{eqn:Gamma1stder}
        \sup_{x \in [0,1)}\left| (\Gamma^{\phi_j}_{\textbf{u}})'(x) - (\Gamma^{\phi_j}_{\textbf{v}})'(x) \right| \geq \rho_0 b^{-Q_0 n},
    \end{equation}
    and
    \begin{equation}\label{eqn:Gamma1stder2}
        \sup_{\substack{I \in \mathcal{L}_{\ell_0}\\I \subset [0,1)}} \inf_{x \in I}\left| (\Gamma^{\phi_j}_{\textbf{u}})'(x) - (\Gamma^{\phi_j}_{\textbf{v}})'(x) \right| \geq \rho_0 \sup_{x \in [0,1]} \left| (\Gamma^{\phi_j}_{\textbf{u}})'(x) - (\Gamma^{\phi_j}_{\textbf{v}})'(x) \right|.
    \end{equation}
\end{theorem}

Here, $\Gamma^{\phi}_{\textbf{u}}(x)$ is defined by (\ref{def:Gamma}), sometimes written as $\Gamma_{\textbf{u}}(x)$ for convenience.

\subsection{The partitions $\mathcal{L}_i^{\mathcal{X}}$}\label{sec:partitionX}
As done in \cite{ren2021dichotomy}, let us define
$$\mathcal{X}=\left\{\pi_{\textbf{j}}\circ g_{\textbf{i}}\::\:\textbf{j}\in\Sigma,\,\textbf{i}\in\Lambda^{\#}\right\}.$$
In this subsection, we establish a nested sequence of partitions $\mathcal{L}_i^{\mathcal{X}}$ of the space $\mathcal{X}$ following the approach in \cite[Section 6]{ren2021dichotomy}.
We then outline several key properties of these partitions.

Recall that for any $\textbf{j}\in\Sigma$ and $\textbf{i}\in\varLambda^\#$, we have:
$$\pi_{\textbf{j}}g_{\textbf{i}}(x,y) = \lambda^{|\textbf{i}|}(y-\Gamma_{{\textbf{i}}^*\textbf{j}}(x)) + \pi_{\textbf{j}}g_{\textbf{i}}(0,0).$$
Thus, each element of $\mathcal{X}$ can be expressed as $\lambda^{t}(y-\psi(x))+c$, where $t\in\mathbb{N}$, $c\in\mathbb{R}^d$, and $\psi(x):\mathbb{R}\to\mathbb{R}^d$ is an analytic function with $\psi(0)=0$. We refer to $|\textbf{i}|$ as the \textit{height} of the map $\pi_{\textbf{j}}g_{\textbf{i}}$.
Define $\overline{\pi}:\mathcal{X}\to\mathbb{N}\times(\mathbb{R}^d)^{M+1}$ as follows:
$$\lambda^{t}(y-\psi(x))+c \mapsto \left(t,\,\psi(\frac{1}{M}),\,\psi(\frac{2}{M}),\ldots,\;\psi(1),\;c\right),$$
where $M=b^{\ell_0}$ and $\ell_0$ is obtained from \S \ref{sec:separation}.
Recall $\mathcal{L}_n$ is the $n$-th $b$-adic partition of  $\mathbb{R}^d$.

\begin{definition} 
   For each integer $n \geq 1$, $\mathcal{L}_n^{\mathcal{X}}$ consists of non-empty subsets of $\mathcal{X}$ given by:
$$\overline{\pi}^{\,-1}\left(\{t\} \times I_1 \times I_2 \times \ldots \times I_M \times J\right),$$
where $t \in \mathbb{N}$, $I_1, I_2, \ldots, I_M \in \mathcal{L}_n$, and $J \in \mathcal{L}_{n+[t\log_b(1/\lambda)]}$.

The partition $\mathcal{L}_0^{\mathcal{X}}$ consists of non-empty subsets of $\mathcal{X}$ represented as:
$$\overline{\pi}^{\,-1}\left(\{t\} \times \mathbb{R}^d \times \ldots \times \mathbb{R}^d \times J\right),$$
where $t \in \mathbb{N}$ and $J \in \mathcal{L}_{[t\log_b(1/\lambda)]}$.
\end{definition}

 The following is a generalization of \cite[Lemma 6.1]{ren2021dichotomy}, which can be proven using the same method.
\begin{lemma}\label{lem:boundspart}
	There exists $A>0$ such that any $i\ge 0$, each element of $\mathcal{L}_i^{\mathcal{X}}$ contains at most $A$ elements of $\mathcal{L}_{i+1}^{\mathcal{X}}$.
\end{lemma}

Recall that for all $x \in \mathbb{R}$ and $\textbf{i} \in \Sigma$,  we have
\begin{equation}\label{eq:Gamma}
\Gamma^{\phi}_{\textbf{i}}(x) = \left(\Gamma^{\phi_1}_{\textbf{i}}(x), \ldots, \Gamma^{\phi_d}_{\textbf{i}}(x)\right).
\end{equation}

\begin{lemma}\label{lem:constantR}
	There exists $K>0$ such that if $\pi_{\textbf{j}}g_{\textbf{u}}$ and $\pi_{\textbf{j}}g_{\textbf{v}}$ belong to the same element of $\mathcal{L}_i^{\mathcal{X}}$, where $\textbf{j}\in \Sigma$, $\textbf{u}, \textbf{v}\in \varLambda^{\hat{\ell}}$, and $i\ge 1$, then for any $x\in [0,1)$ and $y\in \R^d$,
	$$|\pi_{\textbf{j}}g_{\textbf{u}}(x,y)-\pi_{\textbf{j}}g_{\textbf{v}}(x,y)|\le K b^{-(\ell+i)}.$$
\end{lemma}

\begin{proof}
 Combining (\ref{eq:Gamma}) with Theorem \ref{thm:estimate}, we establish the following for each $1\le j\le d$:
\begin{equation}\label{eqn:Gamma'diff}
\sup_{x\in [0,1)}\left|(\Gamma^{\phi_j}_{{\textbf{u}}^*\textbf{j}})'(x)-(\Gamma^{\phi_j}_{{\textbf{v}}^*\textbf{j}})'(x)\right|\le 2\rho_0^{-1}M b^{-i}.
\end{equation}
by employing techniques akin to \cite[(6.1)]{ren2021dichotomy}. Hence, we can conclude the proof using methods similar to \cite[Lemma 6.2]{ren2021dichotomy}.
\end{proof}

By combining \eqref{eqn:Gamma'diff} with Theorem \ref{thm:estimate}, we can prove the following lemma using methods similar to \cite[Lemma 6.3]{ren2021dichotomy}.
\begin{lemma}\label{lem:constantC}
   There exists a constant $C \in \mathbb{Z}_+$ such that for any $\textbf{u} \neq \textbf{v} \in \varLambda^{\ell}$, $\ell \geq 1$, and $\textbf{j} \in \Sigma$, we have $\mathcal{L}_{C\ell}^{\mathcal{X}}(\pi_{\textbf{j}}g_{\textbf{u}}) \neq \mathcal{L}_{C\ell}^{\mathcal{X}}(\pi_{\textbf{j}}g_{\textbf{v}})$.
\end{lemma}

Following \cite{ren2021dichotomy}, let $\xi$ be a discrete probability measure in the space $\mathcal{X}$ and $\mu$ be a Borel probability measure in $\mathbb{R}^d$. Define $\xi \cdot \mu$ as the Borel probability measure in $\mathbb{R}^d$ such that for any Borel subset $A$ of $\mathbb{R}^d$,
$$
\xi \cdot \mu(A) = \xi \times \mu\left(\left\{(\Psi, x) \in \mathcal{X} \times \mathbb{R}^d \::\: \Psi(x) \in A\right\}\right).
$$
For each $\mathbf{u}\in \Lambda^{n}$, set $\mathbf{u}(0)=\frac{u_1+\ldots+u_{n}b^{n-1}}{b^{n}}$. By directly calculating and $\phi(0)=0$, we have
\begin{equation}\nonumber
g_{\mathbf{u}}(0,0)=\left(\,\mathbf{u}(0),W\left(\mathbf{u}(0)\right)\,\right).
\end{equation}
By utilizing this alongside (\ref{eq:Gamma}) and Theorem \ref{thm:estimate}, we can follow \cite[Lemma 6.4]{ren2021dichotomy} to prove the subsequent lemma.

\begin{lemma}\label{lem:entinceta} 
For any $r > 0$, there exist $p =p(r)> 0,\,p_1=p_1(r) > 0$ such that the following holds if $i$ and $n$ are sufficiently large. Let $\xi$ be a probability measure supported in an element of $\mathcal{L}_i^{\mathcal{X}}$ such that each element in the support of $\xi$ has height $\hat{\ell}$, and suppose
	$$\frac{1}{n} H\big(\xi, \mathcal{L}_{i+n}^{\mathcal{X}}\big) > r.$$
Then
	$$\nu^{\,\hat{i}}\left(\left\{\textbf{u} \in \Lambda^{\hat{i}} \::\: \frac{1}{n} H\big(\xi \cdot \big(\delta_{g_{\textbf{u}}(0,0)}\big), \mathcal{L}_{i+n+\ell}\big) \geq p_1\right\}\right) > p.$$
\end{lemma}

\subsection{Positive local entropy for $\theta_{\ell}^{\,\textbf{j}}$}
We need to introduce a discrete measure 
\[
\theta^{\,\textbf{j}}_{\ell}=\frac{1}{b^{\hat{\ell}}}\sum_{\textbf{i}\in \varLambda^{\hat{\ell}}} \delta_{\pi_{\textbf{j}}g_{\textbf{i}}}
\in\boldsymbol{\mathscr{P}}(\mathcal{X})
\]
for each $\ell\in \mathbb{Z}_+$ and analyze the entropy of $\theta^{\,\textbf{j}}_{\ell}$ with respect to the partitions $\mathcal{L}_i^{\mathcal{X}}$ and also the entropy of 
\[
\pi_{\,\textbf{j}}\mu=\theta_{\ell}^{\textbf{j}}.\mu
\]
with respect to the partitions $\mathcal{L}_i$.
We start with analyzing the entropy of $\theta_{\ell}^{\,\textbf{j}}$ with respect to the partitions $\mathcal{L}_i^{\mathcal{X}}$.

\begin{lemma}\label{lem:thetanL0} 
	For $\nu^{\,\mathbb{Z}_+}$-a.e. $\textbf{j}\in \Sigma$, 
\begin{equation}\label{eq:Q.1}
\lim_{\ell\to\infty}\frac{1}{\ell}{H}\left(\theta^{\,\textbf{j}}_{\ell},\mathcal{L}_0^{\mathcal{X}}\right)=\lim_{\ell\to\infty} \frac{1}{\ell} H(\pi_{\textbf{j}}\mu, \mathcal{L}_{\ell})=\alpha.
\end{equation}
\end{lemma}
\begin{proof}
Define $\pi_{\ell}, \pi: \Sigma\;\to\mathbb{R}^{d+1}$, by $\pi_{\ell}(\textbf{i})= g_{i_1i_2\cdots i_{\hat{\ell}}}(0,0)$ and $\pi(\textbf{i})=\lim_{\ell\to\infty} \pi_{\ell}(\textbf{i})$.
Direct calculation shows that $\|\pi_{\ell}-\pi\|_{\infty}=O(b^{-\ell})$. By utilizing this and employing the same methods as in \cite[Lemma 7.1]{ren2021dichotomy}, the claim holds.
\end{proof}

\begin{lemma}\label{lem:thetajLcn}
	There exists $C\in \mathbb{Z}_+$ such that for each $\textbf{j}\in \Sigma$, we have 	$$\lim_{\ell\to\infty}\frac{1}{\ell}H\left(\theta^{\,\textbf{j}}_{\ell},\mathcal{L}_{C \ell}^{\mathcal{X}}\right)=\frac{\log b}{\log\lambda^{-1}}.$$
\end{lemma}
\begin{proof}
	By Lemma~\ref{lem:constantC}, there exists $C\in \mathbb{Z}_+$ such that for all $\ell\ge 1$ and any two distinct $\textbf{i}, \textbf{k}\in \varLambda^{\hat{\ell}}$, $\pi_{\textbf{j}} g_{\textbf{i}}$ and $\pi_{\textbf{j}}g_{\textbf{k}}$ lie in distinct elements of $\mathcal{L}_{C\ell}^{\mathcal{X}}$. Therefore, $H(\theta^{\,\textbf{j}}_{\ell}, \mathcal{L}_{C\ell}^{\mathcal{X}})=\hat{\ell}\log b$. Since $\lim_{\ell\to\infty} \ell/\hat{\ell}=\log_b \lambda^{-1}$, the lemma follows.
\end{proof}

For the remainder, let $\mathbf{j} \in \Sigma$ be chosen to satisfy the conclusions of Lemmas \ref{lem:thetanL0} and \ref{lem:thetajLcn}$,$ and define $\theta_{\ell} = \theta_{\ell}^{\,\mathbf{j}}$.

% and set
%\begin{equation}\label{eqn:eps0}
%\eps_0 = \frac{1}{C} \left(\frac{\log b}{\log \lambda^{-1}} - \alpha \right) > 0.
%\end{equation}

\subsection{Decomposition of entropy}
In the following lemma, we decompose the entropy of $\theta_{\ell}$ and $\pi_{\mathbf{j}}\mu$ into small scales.

\begin{lemma}\label{lem:thetadec}
	 For any $\tau > 0$, there exists $C'(\tau) > 0$ such that if positive integers $n$ and $\ell$ satisfy $\ell > C'(\tau)n$, then
	\begin{equation}\label{eqn:thetadec}
	\frac{1}{C\ell}H\left(\theta_{\ell}, \mathcal{L}_{C\ell}^{\mathcal{X}}\mid\mathcal{L}_0^{\mathcal{X}}\right)\le \mathbb{E}_{0\le i<C\ell}^{\theta_{\ell}} \left[\frac{1}{n}H\left((\theta_{\ell})_{\Psi,\,i}, \mathcal{L}_{i+n}^{\mathcal{X}}\right)\right]+\tau,
	\end{equation}
	\begin{equation}\label{eqn:thetamudec}
	\frac{1}{C\ell} H\left(\pi_{\textbf{j}}\mu, \mathcal{L}_{(C+1)\ell}\mid\mathcal{L}_{\ell}\right)\ge \mathbb{E}^{\theta_{\ell}}_{0\le i<C\ell}\left[\frac1{b^{\hat{i}}}\sum_{\textbf{u}\in\Lambda^{\hat{i}} }\frac{1}{n} H\left((\theta_{\ell})_{\Psi,\,i}.g_{\textbf{u}}\mu, \mathcal{L}_{i+n+\ell}\mid\mathcal{L}_{i+\ell}\right)\right]-\tau.
	\end{equation}
\end{lemma}

\begin{proof}
Utilizing Lemma~\ref{lem:boundspart} and following the similar argument as in \cite[Lemma 3.4]{hochman2014self}, we can prove (\ref{eqn:thetadec}) and the subsequent results when $\ell/n$ is sufficiently large:
\begin{equation}\label{eq:E.4.1}
\frac{1}{C\ell} H\left(\pi_{\textbf{j}}\mu, \mathcal{L}_{(C+1)\ell}\mid\mathcal{L}_{\ell}\right)\ge \mathbb{E}^{\theta_{\ell}}_{0\le i<C\ell}\left[\frac{1}{n} H\left((\theta_{\ell})_{\Psi,\,i}.\mu, \mathcal{L}_{i+n+\ell}\mid\mathcal{L}_{i+\ell}\right)\right]-\tau.
\end{equation}
Note that $(\theta_{\ell})_{\Psi,\,i}.\mu= \frac1{b^{\hat{i}}}\sum_{\textbf{u}\in\Lambda^{\hat{i}} }(\theta_{\ell})_{\Psi,\,i}.g_{\textbf{u}}\mu$. 
By the concavity, we have
$$    H\left((\theta_{\ell})_{\Psi,\,i}.\mu, \mathcal{L}_{i+n+\ell}\mid\mathcal{L}_{i+\ell}\right)\ge \frac1{b^{\hat{i}}}\sum_{\textbf{u}\in\Lambda^{\hat{i}} } H\left((\theta_{\ell})_{\Psi,\,i}.g_{\textbf{u}}\mu, \mathcal{L}_{i+n+\ell}\mid\mathcal{L}_{i+\ell}\right).
 $$
Combining this with (\ref{eq:E.4.1}),	 
thus (\ref{eqn:thetamudec}) holds.
\end{proof}

Let $\ell \in \mathbb{Z}+$ and $i \in \mathbb{N}$ be given. Consider $\Psi \in \mathcal{X}$ such that $\theta_{\ell}(\mathcal{L}_i^{\mathcal{X}}(\Psi)) > 0$. For any $n \in \mathbb{Z}_+$, we define the function $Z_n(\Psi, i,\ell)$ as follows:
\begin{equation}\label{def:Z_n}
Z_{n}(\Psi, i,\ell) = \frac{1}{b^{\hat{i}}} \sum_{\textbf{u}\in\Lambda^{\hat{i}}} \frac{1}{n} H\left( \left( (\theta_{\ell})_{\Psi,\,i} \cdot \delta_{g_{\textbf{u}}(0,0)} \right) \ast \left( \Psi g_{\textbf{u}}\mu \right), \mathcal{L}_{i+n+\ell} \right).
\end{equation}

\begin{lemma}\label{lem:dot2convolution}
	There is a constant $C'>0$ and for each $\tau>0$ there exists $K(\tau)$ such that when $n\ge K(\tau)$, $i\ge C'n$, the following holds for each $\Psi\in\mathcal{X}$ such that $\theta_{\ell}(\mathcal{L}_i^{\mathcal{X}}(\Psi))>0$:
	$$\left|\frac1{b^{\hat{i}}}\sum_{\textbf{u}\in\Lambda^{\hat{i}}}\frac{1}{n} H( (\theta_{\ell})_{\Psi,\,i}. g_{\textbf{u}}\mu,\mathcal{L}_{i+n+\ell}\mid\mathcal{L}_{i+\ell}) - Z_n(\Psi, i,\ell)\right| <\tau.$$
\end{lemma}

\begin{proof}
Given $\tau^{-1}\ll n\ll i$, utilizing Lemma \ref{lem:constantR} and employing similar techniques as in \cite[Lemma 7.7]{ren2021dichotomy}, we establish the following inequality for each $\textbf{u}\in\Lambda^{\hat{i}}$:
$$\left|\frac{1}{n} H(\eta. g_{\textbf{u}}\mu,\mathcal{L}_{i+n+\ell}|\mathcal{L}_{i+\ell}) - \frac{1}{n}H((\eta.\delta_{g_{\textbf{u}}(0,0)})\ast (\Psi g_{\textbf{u}}\mu),\mathcal{L}_{i+n+\ell})\right| <\tau,$$
where $\eta=  (\theta_{\ell})_{\Psi,\,i}$.
By combining this result with (\ref{def:Z_n}), our claim is thus established.
\end{proof}

\begin{lemma}\label{lem;llow}
           For every $\tau > 0$, $n \geq K(\tau)\ge1$, and $i \geq C'n$, the following inequality holds for all $\ell \geq 1$ and $\Psi \in \mathcal{X}$ such that $\theta_{\ell}(\mathcal{L}_i^{\mathcal{X}}(\Psi)) > 0$:
	$$Z_n(\Psi, i,\ell)\ge\alpha-\tau.$$
\end{lemma}

\begin{proof}
Let $n\ll i$ and $\eta = (\theta_{\ell})_{\Psi,\,i}$. Utilizing Lemma \ref{lem:trivialC} and the definition of $Z_n$, we obtain:
	\begin{equation}\label{ap1}
	Z_n(\Psi, i,\ell)\ge\sum_{\Psi'\in\text{supp}(\eta) }\eta(\{\Psi'\})\cdot\frac1{b^{\hat{i}}}\sum_{\textbf{u}\in\Lambda^{\hat{i}} }\frac{1}{n}H\left(\Psi'
	g_{\textbf{u}}\mu,\mathcal{L}_{i+n+\ell}\right)-O(\frac1n).
	\end{equation}
Considering $\Psi' \in \text{supp}(\eta)$ and  (\ref{TransformA}), we find $\Psi'g_{\textbf{u}}(x,y) = \lambda^{\hat{\ell}+\hat{i}}\pi_{\textbf{u}^*\textbf{i}}(x,y) + \text{Constant}$ for some $\textbf{i} \in \Sigma$. Combining this with (\ref{ap1}) and $n \ll i$, we obtain:
\begin{equation}\nonumber
\begin{split}
Z_n(\Psi, i,\ell) & \geq \sum_{\Psi' \in \text{supp}(\eta)} \eta(\{\Psi'\}) \cdot \frac{1}{b^{\hat{i}}} \sum_{\textbf{u} \in \Lambda^{\hat{i}}} \int \frac{1}{n} H\left(\pi_{\textbf{u}^*\textbf{a}}\mu,\mathcal{L}_{n}\right) d\nu^{\,\mathbb{Z}_+}(\textbf{a}) - O(\frac{1}{n}) \\
& = \int \frac{1}{n} H\left(\pi_{\textbf{a}}\mu,\mathcal{L}_{n}\right) d\nu^{\,\mathbb{Z}_+}(\textbf{a}) - O(\frac{1}{n}).
\end{split}
\end{equation}
Thus, our claim holds by Theorem \ref{thm:ledrappier}.
\end{proof}

\subsection{Proof of Theorem C}
The following presents the weak law of large numbers for the entropy of measures $\pi_{\textbf{i}}\mu$.
\begin{lemma}\label{lem:boundsinm}
	For any $\varepsilon>0$, $ n\ge N(\varepsilon), i\ge I(\varepsilon,n)$, 
	$$\inf_{\textbf{a}\in\Sigma}\mathbb{\nu}^{\,i}\left(\left\{\textbf{u}\in \varLambda^{i}\::\: \left|\frac{1}{n}
	H(\pi_{\textbf{u}\textbf{a}}\mu, \mathcal{L}_{n})-\alpha\right|<\varepsilon \right\}\right) >1-\varepsilon.$$
\end{lemma}

\begin{proof}
Given  linear space $A<\mathbb{R}^d$ with $\text{dim }A=1$, utilizing (\ref{def:H}), Theorem \ref{thm:dicho}, and \cite[Lemma 4.1]{ren2021dichotomy}, we find that the measure $\pi_A(\pi^{\phi}_{\textbf{u}}\mu^{\phi})=\pi^{\,\pi_A\phi}_{\textbf{u}}(\mu^{\,\pi_A\phi})$ has no atom for each $\textbf{u}\in\Sigma$. Hence, for any interval $I\subset \mathbb{R}^d$, the function $\textbf{u}\mapsto \pi^{\phi}_{\textbf{u}}\mu^{\phi}(I)$ is continuous. Consequently, we establish the claim using methods similar to \cite[Lemma 4.2]{ren2021dichotomy}.
\end{proof}

\begin{corollary}\label{cor:positive}
	There exist positive real numbers $r$ and $p_2$ such that for all integers $1\ll n\ll \ell$, the following holds:
	$$\mathbb{P}_{0\le i<C\ell}^{\theta_{\ell}} \left(\frac{1}{n}H\left((\theta_{\ell})_{\Psi,\,i}, \mathcal{L}_{i+n}^{\mathcal{X}}\right)>r\right)>p_2.$$
\end{corollary}

\begin{proof}
	Combining   (\ref{eqn:thetadec}),  Lemma \ref{lem:thetanL0}, Lemma \ref{lem:thetajLcn} and $1\ll n\ll \ell$,  we have:
\begin{equation}\label{eq:F.1.1}
	\mathbb{E}_{0\le i<C\ell}^{\theta_{\ell}} \left(\frac{1}{n}H\left((\theta_{\ell})_{\Psi,\,i}, \mathcal{L}_{i+n}^{\mathcal{X}}\right)\right)\ge\frac1{2C}(\log_{\lambda^{-1}}b-\alpha).
\end{equation}
Note that by Lemma \ref{lem:boundspart}, $\frac{1}{n}H\left((\theta_{\ell})_{\Psi,\,i}, \mathcal{L}_{i+n}^{\mathcal{X}}\right) \le \log A$. Combining this inequality with (\ref{eq:F.1.1}) and using probability theory, our claim holds for some constants $r, p_2 > 0$.
\end{proof}

\begin{proof}[Proof of Theorem C]
Let $K > 0$ be determined by Lemma \ref{lem:constantR}, and let $r, p_2 > 0$ be from Corollary \ref{cor:positive}. Choose $p = p(r)$ and $p_1 = p_1(r) > 0$ according to Lemma \ref{lem:entinceta}. Define $R = 4K/\lambda^2$ and $p_0 = p_1/2$. For any $\delta > 0$, select $\tau > 0$ such that $\tau \ll \min\{1,r, p, p_1, p_2, \delta\}$. Choose integers $I, n, \ell > 0$ satisfying $\tau^{-1} \ll n \ll I \ll \ell$.

Given $\tau^{-1} \ll \ell$ and the fact that (\ref{eq:Q.1}) holds for $\textbf{j}$, therefore 
\begin{equation}\nonumber
\frac{1}{C\ell} H(\pi_{\textbf{j}}\mu, \mathcal{L}_{(C+1)\ell}\mid\mathcal{L}_{\ell})\le \alpha+\tau.
\end{equation}	
Combining this with  (\ref{eqn:thetamudec}) and Lemma \ref{lem:dot2convolution}, we have
\begin{equation}\nonumber
\mathbb{E}^{\theta_{\ell}}_{0\le i<C\ell}\left(Z_n\left(\Psi, i,\ell\right)\right)\le \alpha+ O(\tau).
\end{equation}
 Combining this with  Lemma \ref{lem;llow} with Lemma \ref{lem:P}, we have
\begin{equation}\nonumber
\mathbb{P}^{\theta_{\ell}}_{0\le i<C\ell}\left(\,Z_n(\Psi, i,\ell)\le\alpha+O(\tau^{1/2})\,\right)\ge1-O(\tau^{1/2}).
\end{equation}
Combining this with  Corollary  \ref{cor:positive}, we can fix a $\Psi \in\text{supp}(\theta_{\ell})$ and $I\le i\le n$ such that
$$
  Z_n(\Psi, i,\ell)\le\alpha+O(\tau^{1/2})  \quad\quad\text{and}  \quad\quad    \frac{1}{n}H\left((\theta_{\ell})_{\Psi,\,i},\mathcal{L}_{i+n}^{\mathcal{X}}\right)>r.
$$

Let $\xi=(\theta_{\ell})_{\Psi,\,i}$ and $\Psi(x,y)=\lambda^{\hat{\ell}}\pi_{\textbf{i}}(x,y)+Constant$. Employing a similar argument as in Lemma \ref{lem;llow}, it follows that for each $\textbf{u}\in\Lambda^{\hat{i}}$,
\begin{equation}\label{eq:Q.3}
  \frac{1}{n}H\left((\xi.\delta_{g_{\textbf{u}}(0,0)})\ast (\Psi g_{\textbf{u}}\mu),\mathcal{L}_{i+n+\ell}\right)\ge\frac{1}{n}H(\pi_{\textbf{u}^*\textbf{i}}\mu,\mathcal{L}_{n})-O(\frac{1}{n}).
\end{equation}
Combining Lemma \ref{lem:boundsinm} with the definition of $\nu^{\,\hat{i}}$, we obtain
\begin{equation}\label{eq:Q.4}
\mathbb{\nu}^{\,\hat{i}}\left(\left\{\textbf{u}\in \varLambda^{\hat{i}}\::\:\left| \frac{1}{n}
	H(\pi_{\textbf{u}^*\textbf{i}}\mu, \mathcal{L}_{n})-\alpha\right|<\tau \right\}\right) >1-\tau.
\end{equation}
Combining this with \eqref{eq:Q.3}, we find
\begin{equation}\nonumber
\mathbb{\nu}^{\,\hat{i}}\left(\left\{\textbf{u}\in \varLambda^{\hat{i}}\::\:
 \frac{1}{n}H\left((\xi.\delta_{g_{\textbf{u}}(0,0)})\ast (\Psi g_{\textbf{u}}\mu),\mathcal{L}_{i+n+\ell}\right)
     \ge \alpha-O(\tau)\right\}
     \right) >1-O(\tau).
\end{equation}
Furthermore, by combining $Z_n(\Psi, i,\ell)\le\alpha+O(\tau^{1/2})$  and Lemma \ref{lem:P}, we have
\begin{equation}\label{ap12}
\mathbb{\nu}^{\,\hat{i}}\left(\left\{\textbf{u}\in \varLambda^{\hat{i}}\::\:
 \frac{1}{n}H\left((\xi.\delta_{g_{\textbf{u}}(0,0)})\ast (\Psi g_{\textbf{u}}\mu),\mathcal{L}_{i+n+\ell}\right)
     \le \alpha+O(\tau^{1/4})\right\}
     \right) >1-O(\tau^{1/4}).
\end{equation}

Using $\frac{1}{n}H\left(\xi,\mathcal{L}_{i+n}^{\mathcal{X}}\right)>r$ and Lemma \ref{lem:entinceta}, we have
    \begin{equation}\nonumber
      \mathbb{\nu}^{\,\hat{i}}\left(\left\{\textbf{u}\in \varLambda^{\hat{i}}\::\:
 \frac{1}{n}H\left(\xi.\delta_{g_{\textbf{u}}(0,0)},\mathcal{L}_{i+n+\ell}\right)
     \ge p_1\right\}
     \right) >p.
     \end{equation}
Combining this with (\ref{eq:Q.4}) and (\ref{ap12}), there exists $\textbf{u}\in \varLambda^{\hat{i}}$ such that
$$  \frac{1}{n}H\left((\xi.\delta_{g_{\textbf{u}}(0,0)})\ast (\Psi g_{\textbf{u}}\mu),\mathcal{L}_{i+n+\ell}\right)
     \le \alpha+O(\tau^{1/4}),$$
$$\left| \frac{1}{n}
	H(\pi_{\textbf{u}^*\textbf{i}}\mu, \mathcal{L}_{n})-\alpha\right|<\tau \quad\text{and}\quad \frac{1}{n}H\left(\xi.\delta_{g_{\textbf{u}}(0,0)},\mathcal{L}_{i+n+\ell}\right)
     \ge p_1. $$
This implies that  (C.1), (C.2), (C.3) holds for
 $$\eta:=\lambda^{-\hat{\ell}-\hat{i}}(\xi.\delta_{g_{\textbf{u}}(0,0)})\quad\text{and}\quad\textbf{z}:=\textbf{u}^*\textbf{i}.$$
Finally   Lemma \ref{lem:constantR} implies  $\text{diam}\left(\text{supp}(\eta)\right)\le R$.
\end{proof}

\bibliographystyle{plain}             %    other options: alpha, plain, abbrv ...

\end{document}